\documentclass[11pt,a4paper]{article}

\usepackage{geometry}
\usepackage[T1]{fontenc}
\usepackage[utf8]{inputenc}
\usepackage{authblk}

\usepackage{graphics}
\usepackage{tikz,pgf}
\usetikzlibrary{decorations.markings, decorations.pathmorphing}
\usetikzlibrary{arrows,decorations,backgrounds,scopes,plotmarks}
\usepackage{subfigure}
\usepackage{here}
\usepackage{caption}
\usepackage{float}

\usepackage[english]{babel}
\usepackage[all]{xy}
\usepackage{marvosym}
\usepackage{stmaryrd,mathrsfs,amsmath,amssymb,bm,amsfonts}
\usepackage{enumerate}
\usepackage{amsmath,amsthm,amssymb,amscd}

\newtheorem{thm}{Theorem}[section]
\newtheorem{defn}[thm]{Definition}
\newtheorem{prop}[thm]{Proposition}
\newtheorem{lem}[thm]{Lemma}

\newtheorem{ex}[thm]{Example}

\begin{document}

\title{A composition theory for upward planar orders}

\author[a]{Xue Dong}

\author[a]{Xuexing Lu}

\author[b,c]{Yu Ye}

\affil[a]{\small School of Mathematics and Statistics, Zaozhuang University}
\affil[b]{\small Department of Mathematics, USTC}
\affil[c]{Wu Wen-Tsun key laboratory of Mathematics, CAS}

\renewcommand \Authands{ and }
\maketitle

\begin{abstract}
An upward planar order on an acyclic directed graph $G$ is a special linear extension of the edge poset of $G$ that satisfies the nesting condition. This order was introduced to combinatorially characterize upward plane graphs and progressive plane graphs (commonly known as plane string diagrams). In this paper, motivated by the theory of graphical calculus for monoidal categories, we establish a composition theory for upward planar orders. The main result is that the composition of upward planar orders is an upward planar order. This theory provides a practical method to calculate the upward planar order of a progressive plane graph or an upward plane graph.
\end{abstract}

\text{\emph{Keywords}: progressive plane graph; upward plane graph; upward planar order; composition}

\tableofcontents
\section{Introduction}

A  drawing of a (directed or undirected) graph in the plane is called \textbf{planar}, if  its edges intersect only at their endpoints. A planar drawing of a directed graph is called \textbf{upward}, if all edges increase monotonically in the vertical direction (or any other fixed direction). A directed graph together with an upward planar drawing is called an \textbf{upward plane graph}. Fig. $1$ shows an example.
\begin{center}
\begin{tikzpicture}[scale=0.25]
	
	\node (v2) at (0.5,2) {};
	\node (v1) at (-3,-3) {};
	\node (v3) at (3,-2.5) {};
	\node (v4) at (7,0) {};
	\node (v8) at (0.5,-2) {};
	\node (v5) at (-0.5,-7.5) {};
	\node (v6) at (3.5,-5.5) {};
	\node (v7) at (6,-8.5) {};
	\draw  (-3,-3) -- (0.5,2)[postaction={decorate, decoration={markings,mark=at position .5 with {\arrowreversed[black]{stealth}}}}];
	\draw  (3,-2.5) -- (0.5,2)[postaction={decorate, decoration={markings,mark=at position .5 with {\arrowreversed[black]{stealth}}}}];
	\draw  (3,-2.5) -- (7,0)[postaction={decorate, decoration={markings,mark=at position .5 with {\arrowreversed[black]{stealth}}}}];
	\draw  (-0.5,-7.5) -- (3.5,-5.5)[postaction={decorate, decoration={markings,mark=at position .5 with {\arrowreversed[black]{stealth}}}}];
	\draw   (6,-8.5) -- (3.5,-5.5)[postaction={decorate, decoration={markings,mark=at position .5 with {\arrowreversed[black]{stealth}}}}];
	\draw  (-0.5,-7.5) -- (-3,-3)[postaction={decorate, decoration={markings,mark=at position .5 with {\arrowreversed[black]{stealth}}}}];
	\draw  (-0.5,-7.5) -- (0.5,-2)[postaction={decorate, decoration={markings,mark=at position .5 with {\arrowreversed[black]{stealth}}}}];
	\draw  (3.5,-5.5)-- (0.5,-2)[postaction={decorate, decoration={markings,mark=at position .5 with {\arrowreversed[black]{stealth}}}}];
	\draw   (6,-8.5) -- (7,0)[postaction={decorate, decoration={markings,mark=at position .5 with {\arrowreversed[black]{stealth}}}}];
	\node (v9) at (2,-11) {};
	\draw  (2,-11) --  (6,-8.5)[postaction={decorate, decoration={markings,mark=at position .5 with {\arrowreversed[black]{stealth}}}}];
	\draw  (2,-11) -- (-0.5,-7.5)[postaction={decorate, decoration={markings,mark=at position .5 with {\arrowreversed[black]{stealth}}}}];
	\draw  (3.5,-5.5) -- (7,0)[postaction={decorate, decoration={markings,mark=at position .5 with {\arrowreversed[black]{stealth}}}}];
	\draw[fill] (v1) circle [radius=0.2];
	\draw[fill] (v2) circle [radius=0.2];
	\draw[fill] (v3) circle [radius=0.2];
	\draw[fill] (v4) circle [radius=0.2];
	\draw[fill] (v5) circle [radius=0.2];
	\draw[fill] (v6) circle [radius=0.2];
	\draw[fill] (v7) circle [radius=0.2];
	\draw[fill] (v8) circle [radius=0.2];
	\draw[fill] (v9) circle [radius=0.2];
	\draw  plot[smooth, tension=.7] coordinates {(7,0) (8,-2.5) (8.5,-5.5) (7.5,-7.5) (6,-8.5)}[postaction={decorate, decoration={markings,mark=at position .5 with {\arrow[black]{stealth}}}}];

	\node (v12) at (14,0.25) {};
	\node (v13) at (13,-11) {};
	\draw[fill] (v12) circle [radius=0.2];
	\draw[fill] (v13) circle [radius=0.2];
	\node (v14) at (11.25,-5.25) {};
	\node (v15) at (15.25,-6.5) {};
	\draw (14,0.25)  -- (11.25,-5.25)[postaction={decorate, decoration={markings,mark=at position .5 with {\arrow[black]{stealth}}}}];
	\draw (11.25,-5.25) -- (13,-11)[postaction={decorate, decoration={markings,mark=at position .5 with {\arrow[black]{stealth}}}}];
	\draw (14,0.25) -- (15.25,-6.5)[postaction={decorate, decoration={markings,mark=at position .5 with {\arrow[black]{stealth}}}}];
	\draw (15.25,-6.5)-- (13,-11)[postaction={decorate, decoration={markings,mark=at position .5 with {\arrow[black]{stealth}}}}];
	\node (v16) at (13.25,-3) {};
	\node (v17) at (12.75,-8) {};
	\draw  (13.25,-3) --(12.75,-8)[postaction={decorate, decoration={markings,mark=at position .5 with {\arrow[black]{stealth}}}}];
	\draw[fill] (v14) circle [radius=0.2];
	\draw[fill] (v15) circle [radius=0.2];
	\draw[fill] (v16) circle [radius=0.2];
	\draw[fill] (v17) circle [radius=0.2];
	\node (v18) at (17.75,-3.5) {};
	\node (v19) at (17,-9) {};
	\draw (17.75,-3.5)-- (17,-9)[postaction={decorate, decoration={markings,mark=at position .5 with {\arrow[black]{stealth}}}}];
	\draw[fill] (v18) circle [radius=0.2];
	\draw[fill] (v19) circle [radius=0.2];
	\draw  plot[smooth, tension=.7] coordinates {(14,0.25) (16.25,-2.25)  (15.25,-6.5)}[postaction={decorate, decoration={markings,mark=at position .5 with {\arrow[black]{stealth}}}}];
\end{tikzpicture}

Figure $1$. An upward plane graph
\end{center}

A directed graph is called an \textbf{upward planar graph} if it admits an upward planar drawing. An upward planar graph is necessarily acyclic. As natural visibility representations of acyclic directed graphs, upward plane graphs are commonly used to represent hierarchical structures, such as PERT networks, Hasse diagrams, family trees, etc., and have been extensively studied in the fields of graph theory, graph drawing algorithms, ordered set theory, etc. Upward planarity has been an active field of graph theory; for a review, see \cite{[GT95]}.

An essentially equivalent notion in the field of graphical calculi is that of a \textbf{progressive plane graph} (or \textbf{plane string diagram}), which is an upward planar drawing of an acyclic directed graph (possibly with isolated vertices) in a plane box satisfying the boundary conditions: $(1)$ all vertices drawn on the horizontal boundaries are leaves (vertices of degree one); $(2)$ no vertices are drawn on the vertical boundaries \cite{[HLY19]}, see Fig. 2 for an example. It was introduced by Joyal and Street \cite{[JS88],[JS91]} as a core notion in the graphical calculus for monoidal categories. A progressive plane graph is just a boxed upward plane graph satisfying the boundary conditions, with the edges intersecting with the top and bottom boundaries as input and output edges, respectively.

\begin{center}
\begin{tikzpicture}[scale=0.45]

\draw [dashed] (-3,1.5) rectangle (8.5,-6.5);
\node [above](v1) at (-1,1.5) {};
\node (v2) at (0,-1) {};
\node (v6) at (-1.5,-4.5) {};
\node [left] at (-1.5,-4.5) {};
\node (v3) at (1.5,0.5) {};
\node (v8) at (4,-2) {};
\node (v4) at (0.5,-3.5) {};
\node [below](v5) at (0.5,-6.5) {};
\node [above](v7) at (4,1.5) {};
\draw  (-1,1.5)  -- (0,-1)[postaction={decorate, decoration={markings,mark=at position .5 with {\arrow[black]{stealth}}}}];
\draw   (1.5,0.5) -- (0,-1)[postaction={decorate, decoration={markings,mark=at position .5 with {\arrow[black]{stealth}}}}];
\draw   (1.5,0.5)--  (0.5,-3.5)[postaction={decorate, decoration={markings,mark=at position .5 with {\arrow[black]{stealth}}}}];
\draw  (0.5,-3.5) -- (v5)[postaction={decorate, decoration={markings,mark=at position .5 with {\arrow[black]{stealth}}}}];
\draw  (0,-1) -- (-1.5,-4.5)[postaction={decorate, decoration={markings,mark=at position .5 with {\arrow[black]{stealth}}}}];
\draw  (0,-1) -- (0.5,-3.5)[postaction={decorate, decoration={markings,mark=at position .5 with {\arrow[black]{stealth}}}}];
\draw  (4,1.5)-- (4,-2)[postaction={decorate, decoration={markings,mark=at position .5 with {\arrow[black]{stealth}}}}];
\draw   (1.5,0.5)-- (4,-2)[postaction={decorate, decoration={markings,mark=at position .5 with {\arrow[black]{stealth}}}}];
\node (v9) at (3,-5.5) {};
\draw  (0.5,-3.5) -- (3,-5.5)[postaction={decorate, decoration={markings,mark=at position .5 with {\arrow[black]{stealth}}}}];
\draw  (4,-2) -- (3,-5.5)[postaction={decorate, decoration={markings,mark=at position .5 with {\arrow[black]{stealth}}}}];
\node (v10) at (2,-3.5) {};
\draw   (1.5,0.5) -- (2,-3.5)[postaction={decorate, decoration={markings,mark=at position .5 with {\arrow[black]{stealth}}}}];
\draw  (4,-2) -- (2,-3.5)[postaction={decorate, decoration={markings,mark=at position .5 with {\arrow[black]{stealth}}}}];

\draw [fill](v2) circle [radius=0.11];
\draw [fill](v3) circle [radius=0.11];
\draw [fill](v4) circle [radius=0.11];
\draw [fill](v6) circle [radius=0.11];
\draw [fill](v8) circle [radius=0.11];
\draw [fill](v9) circle [radius=0.11];
\draw [fill](v10) circle [radius=0.11];

\node at (-1,0.5) {};
\node at (4.5,0) {};
\node at (0,-5.5) {};
\node at (-1.5,-3) {};

\draw(6.5,1.5)--(6.5,-6.5)[postaction={decorate, decoration={markings,mark=at position .5 with {\arrow[black]{stealth}}}}];
\node [above]at(6.5,1.5) {};
\node [below]at (6.5,-6.5) {};
\node at (7,-3) {};

\draw [fill] (-1,1.5) circle [radius=0.11];
\draw [fill](0.5,-6.5) circle [radius=0.11];
\draw [fill](4,1.5) circle [radius=0.11];
\draw [fill](6.5,-6.5) circle [radius=0.11];
\draw [fill](6.5,1.5) circle [radius=0.11];

\draw [fill](5.5,-1) circle [radius=0.11];
\draw [fill](2.7,-1.7) circle [radius=0.11];

\draw [fill](0.35,0.45) circle [radius=0.11];
\end{tikzpicture}

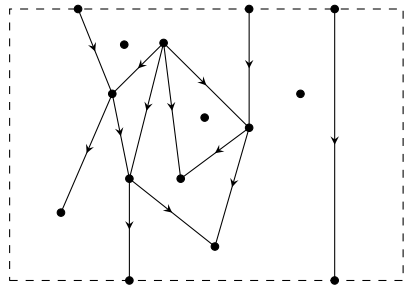
Figure 2. A progressive plane graph with isolated vertices
\end{center}

There are several approaches to characterize upward planarity. The first approach, given independently in \cite{[BT88]} and \cite{[Ke87]},   characterizes  upward planar graphs as  spanning subgraphs of \textbf{planar $st$ graphs}.
The second approach, given in \cite{[BB91],[BBLM94]}, characterizes  upward plane graphs by means of \textbf{$2$-list-embedding}  and \textbf{upward-consistent assignment} of sources and sinks to faces. The third approach, given in \cite{[LY19]}, characterizes upward plane graphs in term of \textbf{upward planar orders}. From the perspective of graphical calculus for monoidal categories, the three approaches are closely related to each other and can be interpreted by the \textbf{unit convention}. A detailed explanation will be given elsewhere.

An upward planar order on an acyclic directed graph $G$ is a special linear extension of the edge poset of $G$ that satisfies the nesting condition (the $Q_2$-condition in Definition \ref{upo1}). The axioms of an upward planar order are reviewed in Definition \ref{upo} and Definition \ref{upo1}. In \cite{[LY19]}, it was shown that there is a natural upward planar order associated with an upward plane graph, see Fig. $3$
for an example, where the linear order of edges represents the associated upward planar order of the upward plane graph in Fig. $1$.

\begin{center}
\begin{tikzpicture}[scale=0.3]
\node (v2) at (0.5,2) {};
\node (v1) at (-3,-3) {};
	\node (v3) at (3,-2.5) {};
	\node (v4) at (7,0) {};
	\node (v8) at (0.5,-2) {};
	\node (v5) at (-0.5,-7.5) {};
	\node (v6) at (3.5,-5.5) {};
	\node (v7) at (6,-8.5) {};
	\draw  (-3,-3) -- (0.5,2)[postaction={decorate, decoration={markings,mark=at position .5 with {\arrowreversed[black]{stealth}}}}];
	\draw  (3,-2.5) -- (0.5,2)[postaction={decorate, decoration={markings,mark=at position .5 with {\arrowreversed[black]{stealth}}}}];
	\draw  (3,-2.5) -- (7,0)[postaction={decorate, decoration={markings,mark=at position .5 with {\arrowreversed[black]{stealth}}}}];
	\draw  (-0.5,-7.5) -- (3.5,-5.5)[postaction={decorate, decoration={markings,mark=at position .5 with {\arrowreversed[black]{stealth}}}}];
	\draw   (6,-8.5) -- (3.5,-5.5)[postaction={decorate, decoration={markings,mark=at position .5 with {\arrowreversed[black]{stealth}}}}];
	\draw  (-0.5,-7.5) -- (-3,-3)[postaction={decorate, decoration={markings,mark=at position .5 with {\arrowreversed[black]{stealth}}}}];
	\draw  (-0.5,-7.5) -- (0.5,-2)[postaction={decorate, decoration={markings,mark=at position .5 with {\arrowreversed[black]{stealth}}}}];
	\draw  (3.5,-5.5)-- (0.5,-2)[postaction={decorate, decoration={markings,mark=at position .5 with {\arrowreversed[black]{stealth}}}}];
	\draw   (6,-8.5) -- (7,0)[postaction={decorate, decoration={markings,mark=at position .5 with {\arrowreversed[black]{stealth}}}}];
	\node (v9) at (2,-11) {};
	\draw  (2,-11) --  (6,-8.5)[postaction={decorate, decoration={markings,mark=at position .5 with {\arrowreversed[black]{stealth}}}}];
	\draw  (2,-11) -- (-0.5,-7.5)[postaction={decorate, decoration={markings,mark=at position .5 with {\arrowreversed[black]{stealth}}}}];
	\draw  (3.5,-5.5) -- (7,0)[postaction={decorate, decoration={markings,mark=at position .5 with {\arrowreversed[black]{stealth}}}}];
	\draw[fill] (v1) circle [radius=0.2];
	\draw[fill] (v2) circle [radius=0.2];
	\draw[fill] (v3) circle [radius=0.2];
	\draw[fill] (v4) circle [radius=0.2];
	\draw[fill] (v5) circle [radius=0.2];
	\draw[fill] (v6) circle [radius=0.2];
	\draw[fill] (v7) circle [radius=0.2];
	\draw[fill] (v8) circle [radius=0.2];
	\draw[fill] (v9) circle [radius=0.2];
	\draw  plot[smooth, tension=.7] coordinates {(7,0) (8,-2.5) (8.5,-5.5) (7.5,-7.5) (6,-8.5)}[postaction={decorate, decoration={markings,mark=at position .5 with {\arrow[black]{stealth}}}}];
	\node [scale=0.68]at (-1.9,0) {$1$};
	\node [scale=0.68]at (-2.5,-5.6) {$2$};
	\node [scale=0.68]at (-0.5,-4.3) {$3$};
	\node [scale=0.68]at (2.1,-2.9) {$4$};
	\node [scale=0.68]at (2.2,0.5) {$5$};
	\node [scale=0.68]at (4.8,-0.7) {$6$};
	\node [scale=0.68]at (4.1,-3.7) {$7$};
	\node [scale=0.68]at (1.9,-7.2) {$8$};
	\node [scale=0.68]at (0.1,-9.7) {$9$};
	\node [scale=0.68]at (4.2,-7.4) {$10$};
	\node[scale=0.68] at (5.7,-5) {$11$};
	\node [scale=0.68]at (9.4,-4.5) {$12$};
	\node [scale=0.68]at (4.4,-10.5) {$13$};
	
	\node (v12) at (14,0.25) {};
	\node (v13) at (13,-11) {};
	\draw[fill] (v12) circle [radius=0.2];
	\draw[fill] (v13) circle [radius=0.2];
	\node (v14) at (11.25,-5.25) {};
	\node (v15) at (15.25,-6.5) {};
	\draw (14,0.25)  -- (11.25,-5.25)[postaction={decorate, decoration={markings,mark=at position .5 with {\arrow[black]{stealth}}}}];
	\draw (11.25,-5.25) -- (13,-11)[postaction={decorate, decoration={markings,mark=at position .5 with {\arrow[black]{stealth}}}}];
	\draw (14,0.25) -- (15.25,-6.5)[postaction={decorate, decoration={markings,mark=at position .5 with {\arrow[black]{stealth}}}}];
	\draw (15.25,-6.5)-- (13,-11)[postaction={decorate, decoration={markings,mark=at position .5 with {\arrow[black]{stealth}}}}];
	\node (v16) at (13.25,-3) {};
	\node (v17) at (12.75,-8) {};
	\draw  (13.25,-3) --(12.75,-8)[postaction={decorate, decoration={markings,mark=at position .5 with {\arrow[black]{stealth}}}}];
	\draw[fill] (v14) circle [radius=0.2];
	\draw[fill] (v15) circle [radius=0.2];
	\draw[fill] (v16) circle [radius=0.2];
	\draw[fill] (v17) circle [radius=0.2];
	\node[scale=0.68] at (12,-2) {$14$};
	\node[scale=0.68] at (11.25,-8.75) {$15$};
	\node[scale=0.68] at (13.75,-5.75) {$16$};
	\node[scale=0.68] at (15.25,-3) {$17$};
	\node [scale=0.68] at (14.75,-9.25) {$19$};
	\node (v18) at (17.75,-3.5) {};
	\node (v19) at (17,-9) {};
	\draw (17.75,-3.5)-- (17,-9)[postaction={decorate, decoration={markings,mark=at position .5 with {\arrow[black]{stealth}}}}];
	\draw[fill] (v18) circle [radius=0.2];
	\draw[fill] (v19) circle [radius=0.2];
	\node [scale=0.68] at (18.25,-6.5) {$20$};
	\draw  plot[smooth, tension=.7] coordinates {(14,0.25) (16.25,-2.25)  (15.25,-6.5)}[postaction={decorate, decoration={markings,mark=at position .5 with {\arrow[black]{stealth}}}}];
	\node[scale=0.68] at (16.75,-1) {$18$};
\end{tikzpicture}

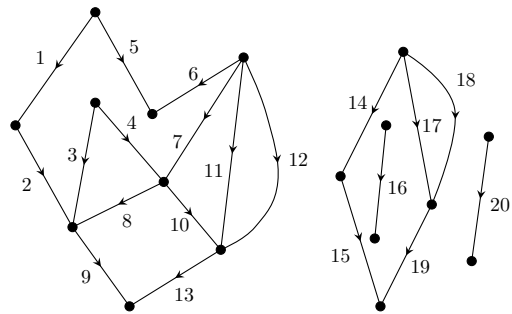
Figure $3$. The upward planar order associated with the upward plane graph in Fig. $1$.
\end{center}

The existence of the associated upward planar order of an upward plane graph has been proved in  \cite{[LY19]}, but in an indirected way.  The next interesting problem in the theory of upward planar order is how to calculate the associated upward planar order directly. In this paper, we take a compositional strategy to solve this problem, which was motivated by the graphical calculus for monoidal categories \cite{[JS91]}, where to calculate the value of a diagram, people have to cut the diagram into elementary layers. We establish a composition theory for upward planar orders, which is a natural generalization of  the composition theory for planar orders on \textbf{processive} graphs and $st$ graph \cite{[HLY19]}.

The idea of composition here is rather simple. Take the upward plane graph in Fig. $1$ as an example. We first cut it in to elementary layers (Fig. $4$). Then we make these layers into elementary progressive plane graphs (Fig. $5$), whose associated upward planar orders can be obtained directly, following the convention that first from left to right and then from up to down (Fig. $6$). The final step is to compose all the upward planar orders of elementary layers into one upward planar order (Fig. $3$).

\begin{center}
\begin{tikzpicture}[scale=0.3]
\node (v2) at (0.5,2) {};
\node (v1) at (-3,-3) {};
\node (v3) at (3,-2.5) {};
\node (v4) at (7,0) {};
\node (v8) at (0.5,-2) {};
\node (v5) at (-0.5,-7.5) {};
\node (v6) at (3.5,-5.5) {};
\node (v7) at (6,-8.5) {};
\draw  (-3,-3) -- (0.5,2)[postaction={decorate, decoration={markings,mark=at position .5 with {\arrowreversed[black]{stealth}}}}];
\draw  (3,-2.5) -- (0.5,2)[postaction={decorate, decoration={markings,mark=at position .5 with {\arrowreversed[black]{stealth}}}}];
\draw  (3,-2.5) -- (7,0)[postaction={decorate, decoration={markings,mark=at position .5 with {\arrowreversed[black]{stealth}}}}];
\draw  (-0.5,-7.5) -- (3.5,-5.5)[postaction={decorate, decoration={markings,mark=at position .5 with {\arrowreversed[black]{stealth}}}}];
\draw   (6,-8.5) -- (3.5,-5.5)[postaction={decorate, decoration={markings,mark=at position .5 with {\arrowreversed[black]{stealth}}}}];
\draw  (-0.5,-7.5) -- (-3,-3)[postaction={decorate, decoration={markings,mark=at position .5 with {\arrowreversed[black]{stealth}}}}];
\draw  (-0.5,-7.5) -- (0.5,-2)[postaction={decorate, decoration={markings,mark=at position .5 with {\arrowreversed[black]{stealth}}}}];
\draw  (3.5,-5.5)-- (0.5,-2)[postaction={decorate, decoration={markings,mark=at position .5 with {\arrowreversed[black]{stealth}}}}];
\draw   (6,-8.5) -- (7,0)[postaction={decorate, decoration={markings,mark=at position .5 with {\arrowreversed[black]{stealth}}}}];
\node (v9) at (2,-11) {};
\draw  (2,-11) --  (6,-8.5)[postaction={decorate, decoration={markings,mark=at position .5 with {\arrowreversed[black]{stealth}}}}];
\draw  (2,-11) -- (-0.5,-7.5)[postaction={decorate, decoration={markings,mark=at position .5 with {\arrowreversed[black]{stealth}}}}];
\draw  (3.5,-5.5) -- (7,0)[postaction={decorate, decoration={markings,mark=at position .5 with {\arrowreversed[black]{stealth}}}}];
\draw[fill] (v1) circle [radius=0.2];
\draw[fill] (v2) circle [radius=0.2];
\draw[fill] (v3) circle [radius=0.2];
\draw[fill] (v4) circle [radius=0.2];
\draw[fill] (v5) circle [radius=0.2];
\draw[fill] (v6) circle [radius=0.2];
\draw[fill] (v7) circle [radius=0.2];
\draw[fill] (v8) circle [radius=0.2];
\draw[fill] (v9) circle [radius=0.2];
\draw  plot[smooth, tension=.7] coordinates {(7,0) (8,-2.5) (8.5,-5.5) (7.5,-7.5) (6,-8.5)}[postaction={decorate, decoration={markings,mark=at position .5 with {\arrow[black]{stealth}}}}];

\node (v12) at (14,0.25) {};
\node (v13) at (13,-11) {};
\draw[fill] (v12) circle [radius=0.2];
\draw[fill] (v13) circle [radius=0.2];
\node (v14) at (11.25,-5.25) {};
\node (v15) at (15.25,-6.5) {};
\draw (14,0.25)  -- (11.25,-5.25)[postaction={decorate, decoration={markings,mark=at position .5 with {\arrow[black]{stealth}}}}];
\draw (11.25,-5.25) -- (13,-11)[postaction={decorate, decoration={markings,mark=at position .5 with {\arrow[black]{stealth}}}}];
\draw (14,0.25) -- (15.25,-6.5)[postaction={decorate, decoration={markings,mark=at position .5 with {\arrow[black]{stealth}}}}];
\draw (15.25,-6.5)-- (13,-11)[postaction={decorate, decoration={markings,mark=at position .5 with {\arrow[black]{stealth}}}}];
\node (v16) at (13.25,-3) {};
\node (v17) at (12.75,-8) {};
\draw  (13.25,-3) --(12.75,-8)[postaction={decorate, decoration={markings,mark=at position .5 with {\arrow[black]{stealth}}}}];
\draw[fill] (v14) circle [radius=0.2];
\draw[fill] (v15) circle [radius=0.2];
\draw[fill] (v16) circle [radius=0.2];
\draw[fill] (v17) circle [radius=0.2];
\node (v18) at (17.75,-3.5) {};
\node (v19) at (17,-9) {};
\draw (17.75,-3.5)-- (17,-9)[postaction={decorate, decoration={markings,mark=at position .5 with {\arrow[black]{stealth}}}}];
\draw[fill] (v18) circle [radius=0.2];
\draw[fill] (v19) circle [radius=0.2];
\draw  plot[smooth, tension=.7] coordinates {(14,0.25) (16.25,-2.25)  (15.25,-6.5)}[postaction={decorate, decoration={markings,mark=at position .5 with {\arrow[black]{stealth}}}}];
\draw[dashed]  (-4.5,2.5) rectangle (19.5,-11.5);

\node (v10) at (-4.4,-9.6) {};
\node (v11) at (19.4,-10) {};
\draw [densely dotted] (-4.4,-9.6)--(19.4,-10);
\node (v20) at (-4.4,-6.2) {};
\node (v21) at (19.4,-7.6) {};
\draw [densely dotted] (-4.4,-6.2) -- (19.4,-7.6);
\node (v22) at (-4.4,-4.2) {};
\node (v23) at (19.4,-4.4) {};
\draw [densely dotted](-4.4,-4.2) --(19.4,-4.4);
\node (v25) at (-4.6,-0.8) {};
\node (v24) at (19.4,-1.6) {};
\draw [densely dotted] (-4.4,-0.8) --(19.4,-1.6);
\end{tikzpicture}

 Figure $4$. Cut into elementary layers
\end{center}

\begin{center}
	\begin{tikzpicture}[scale=0.5]
		
		\draw [dashed] (0,2.3) rectangle (14,0);
		\draw [dashed] (0,-0.5) rectangle (14,-2.5);
		\draw [dashed] (0,-3) rectangle (14,-5);
		\draw [dashed] (0,-5.5) rectangle (14,-7.5);
		\draw [dashed] (0,-8) rectangle (14,-10.25);
		
		\draw  (2,2) rectangle (2,2) node (v1) {};
		\draw  (1.5,0) rectangle (1.5,0) node (v2) {};
		\draw  (3,0) rectangle (3,0) node (v3) {};
		\draw  (4,0) rectangle (4,0) node (v5) {};
		\draw  (5,0) rectangle (5,0) node (v6) {};
		\draw  (6,0) rectangle (6,0) node (v7) {};
		\draw  (7,0) rectangle (7,0) node (v8) {};
		\draw  (10,0) rectangle (10,0) node (v10) {};
		\draw  (11,0) rectangle (11,0) node (v11) {};
		\draw  (12,0) rectangle (12,0) node (v12) {};
		\draw  (5.5,1) rectangle (5.5,1) node (v4) {};
		\draw  (11,1) rectangle (11,1) node (v9) {};
		
		\draw  (2,2)-- (1.5,0)[postaction={decorate, decoration={markings,mark=at position .5 with {\arrow[black]{stealth}}}}];
		\draw  (2,2)-- (3,0)[postaction={decorate, decoration={markings,mark=at position .5 with {\arrow[black]{stealth}}}}];
		\draw  (5.5,1) -- (4,0)[postaction={decorate, decoration={markings,mark=at position .5 with {\arrow[black]{stealth}}}}];
		\draw  (5.5,1) -- (5,0)[postaction={decorate, decoration={markings,mark=at position .5 with {\arrow[black]{stealth}}}}];
		\draw(5.5,1) -- (6,0)[postaction={decorate, decoration={markings,mark=at position .5 with {\arrow[black]{stealth}}}}];
		\draw (5.5,1) -- (7,0)[postaction={decorate, decoration={markings,mark=at position .5 with {\arrow[black]{stealth}}}}];
		\draw (11,1) -- (10,0)[postaction={decorate, decoration={markings,mark=at position .5 with {\arrow[black]{stealth}}}}];
		\draw  (11,1)-- (11,0)[postaction={decorate, decoration={markings,mark=at position .5 with {\arrow[black]{stealth}}}}];
		\draw  (11,1)-- (12,0)[postaction={decorate, decoration={markings,mark=at position .5 with {\arrow[black]{stealth}}}}];
		\draw[fill] (v1) circle [radius=0.1];
		\draw[fill] (v4) circle [radius=0.1];
		\draw[fill] (v9) circle [radius=0.1];

		\node (v29) at (9,-2.5) {};
		\node (v31) at (10,-2.5) {};
		\node (v33) at (11,-2.5) {};
		\node (v35) at (12,-2.5) {};
		\node (v37) at (13,-2.5) {};
		\node (v14) at (1.5,-1.5) {};
		\node (v16) at (2.5,-1.5) {};
		\node (v20) at (4,-1.5) {};
		\node (v30) at (10,-1.5) {};
		\node (v36) at (13,-1.5) {};
		\node (v13) at (1.5,-0.5) {};
		\node (v19) at (3,-0.5) {};
		\node (v21) at (4,-0.5) {};
		\node (v22) at (5,-0.5) {};
		\node (v24) at (6,-0.5) {};
		\node (v26) at (7,-0.5) {};
		\node (v28) at (10,-0.5) {};
		\node (v32) at (11,-0.5) {};
		\node (v34) at (12,-0.5) {};
		\node (v15) at (1.5,-2.5) {};
		\node (v17) at (2.5,-2.5) {};
		\node (v18) at (3.5,-2.5) {};
		\node (v23) at (5,-2.5) {};
		\node (v25) at (6,-2.5) {};
		\node (v27) at (7,-2.5) {};
		\draw (1.5,-0.5) --(1.5,-1.5)[postaction={decorate, decoration={markings,mark=at position .5 with {\arrow[black]{stealth}}}}];
		\draw  (1.5,-1.5) -- (1.5,-2.5)[postaction={decorate, decoration={markings,mark=at position .5 with {\arrow[black]{stealth}}}}];
		\draw  (2.5,-1.5) --  (2.5,-2.5)[postaction={decorate, decoration={markings,mark=at position .5 with {\arrow[black]{stealth}}}}];
		\draw (2.5,-1.5) -- (3.5,-2.5)[postaction={decorate, decoration={markings,mark=at position .5 with {\arrow[black]{stealth}}}}];
		\draw  (3,-0.5)--  (4,-1.5)[postaction={decorate, decoration={markings,mark=at position .5 with {\arrow[black]{stealth}}}}];
		\draw  (4,-0.5)--  (4,-1.5)[postaction={decorate, decoration={markings,mark=at position .5 with {\arrow[black]{stealth}}}}];
		\draw (5,-0.5)--  (5,-2.5)[postaction={decorate, decoration={markings,mark=at position .5 with {\arrow[black]{stealth}}}}];
		\draw  (6,-0.5)-- (6,-2.5)[postaction={decorate, decoration={markings,mark=at position .5 with {\arrow[black]{stealth}}}}];
		\draw  (7,-0.5) -- (7,-2.5)[postaction={decorate, decoration={markings,mark=at position .5 with {\arrow[black]{stealth}}}}];
		\draw  (10,-0.5) -- (9,-2.5)[postaction={decorate, decoration={markings,mark=at position .5 with {\arrow[black]{stealth}}}}];
		\draw   (10,-1.5)-- (10,-2.5)[postaction={decorate, decoration={markings,mark=at position .5 with {\arrow[black]{stealth}}}}];
		\draw (11,-0.5)  -- (11,-2.5)[postaction={decorate, decoration={markings,mark=at position .5 with {\arrow[black]{stealth}}}}];
		\draw (12,-0.5)-- (12,-2.5)[postaction={decorate, decoration={markings,mark=at position .5 with {\arrow[black]{stealth}}}}];
		\draw  (13,-1.5) -- (13,-2.5)[postaction={decorate, decoration={markings,mark=at position .5 with {\arrow[black]{stealth}}}}];
		\draw[fill] (v14) circle [radius=0.1];
		\draw[fill] (v16) circle [radius=0.1];
		\draw[fill] (v20) circle [radius=0.1];
		\draw[fill] (v36) circle [radius=0.1];
		\draw[fill] (v30) circle [radius=0.1];
		\node (v38) at (1.5,-3) {};
		\node (v40) at (2.5,-3) {};
		\node (v42) at (3.5,-3) {};
		\node (v44) at (5,-3) {};
		\node (v47) at (6,-3) {};
		\node (v49) at (7,-3) {};
		\node (v51) at (9,-3) {};
		\node (v54) at (10,-3) {};
		\node (v56) at (11,-3) {};
		\node (v58) at (12,-3) {};
		\node (v60) at (13,-3) {};
		\node (v39) at (1.5,-5) {};
		\node (v41) at (2.5,-5) {};
		\node (v45) at (3.5,-5) {};
		\node (v46) at (4.5,-5) {};
		\node (v48) at (5.5,-5) {};
		\node (v50) at (6.5,-5) {};
		\node (v53) at (9,-5) {};
		\node (v55) at (10,-5) {};
		\node (v59) at (11,-5) {};
		\node (v61) at (13,-5) {};
		\draw  (1.5,-3) -- (1.5,-5)[postaction={decorate, decoration={markings,mark=at position .5 with {\arrow[black]{stealth}}}}];
		\draw (2.5,-3) -- (2.5,-5)[postaction={decorate, decoration={markings,mark=at position .5 with {\arrow[black]{stealth}}}}];
		\node (v43) at (4,-4) {};
		\draw  (3.5,-3) -- (4,-4)[postaction={decorate, decoration={markings,mark=at position .5 with {\arrow[black]{stealth}}}}];
		\draw  (5,-3)-- (4,-4)[postaction={decorate, decoration={markings,mark=at position .5 with {\arrow[black]{stealth}}}}];
		\draw  (4,-4)--  (3.5,-5)[postaction={decorate, decoration={markings,mark=at position .5 with {\arrow[black]{stealth}}}}];
		\draw  (4,-4) -- (4.5,-5)[postaction={decorate, decoration={markings,mark=at position .5 with {\arrow[black]{stealth}}}}];
		\draw  (6,-3) --  (5.5,-5)[postaction={decorate, decoration={markings,mark=at position .5 with {\arrow[black]{stealth}}}}];
		\draw (7,-3)-- (6.5,-5)[postaction={decorate, decoration={markings,mark=at position .5 with {\arrow[black]{stealth}}}}];
		\node (v52) at (9,-4) {};
		\draw  (9,-3) --(9,-4)[postaction={decorate, decoration={markings,mark=at position .5 with {\arrow[black]{stealth}}}}];
		\draw  (9,-4) -- (9,-5)[postaction={decorate, decoration={markings,mark=at position .5 with {\arrow[black]{stealth}}}}];
		\draw(10,-3) -- (10,-5)[postaction={decorate, decoration={markings,mark=at position .5 with {\arrow[black]{stealth}}}}];
		\node (v57) at (11.5,-4) {};
		\draw  (11,-3)--  (11.5,-4)[postaction={decorate, decoration={markings,mark=at position .5 with {\arrow[black]{stealth}}}}];
		\draw  (12,-3)--  (11.5,-4)[postaction={decorate, decoration={markings,mark=at position .5 with {\arrow[black]{stealth}}}}];
		\draw   (11.5,-4)-- (11,-5)[postaction={decorate, decoration={markings,mark=at position .5 with {\arrow[black]{stealth}}}}];
		\draw (13,-3)-- (13,-5)[postaction={decorate, decoration={markings,mark=at position .5 with {\arrow[black]{stealth}}}}];
		\draw[fill] (v43) circle [radius=0.1];
		\draw[fill] (v52) circle [radius=0.1];
		\draw[fill] (v57) circle [radius=0.1];
		\node (v62) at (1.5,-5.5) {};
		\node (v64) at (2.5,-5.5) {};
		\node (v65) at (3.5,-5.5) {};
		\node (v66) at (4.5,-5.5) {};
		\node (v68) at (5.5,-5.5) {};
		\node (v69) at (6.5,-5.5) {};
		\node (v72) at (9,-5.5) {};
		\node (v74) at (10,-5.5) {};
		\node (v76) at (11,-5.5) {};
		\node (v78) at (13,-5.5) {};
		\node (v70) at (3,-7.5) {};
		\node (v71) at (5,-7.5) {};
		\node (v73) at (9.5,-7.5) {};
		\node (v77) at (10.5,-7.5) {};
		\node (v63) at (2.5,-6.5) {};
		\node (v67) at (5.5,-6.5) {};
		\node (v75) at (10,-6.5) {};
		\node (v79) at (13,-6.5) {};
		\draw  (1.5,-5.5)-- (2.5,-6.5)[postaction={decorate, decoration={markings,mark=at position .5 with {\arrow[black]{stealth}}}}];
		\draw (2.5,-5.5) -- (2.5,-6.5)[postaction={decorate, decoration={markings,mark=at position .5 with {\arrow[black]{stealth}}}}];
		\draw  (3.5,-5.5) --(2.5,-6.5)[postaction={decorate, decoration={markings,mark=at position .5 with {\arrow[black]{stealth}}}}];
		\draw  (4.5,-5.5) -- (5.5,-6.5)[postaction={decorate, decoration={markings,mark=at position .5 with {\arrow[black]{stealth}}}}];
		\draw  (5.5,-5.5) -- (5.5,-6.5)[postaction={decorate, decoration={markings,mark=at position .5 with {\arrow[black]{stealth}}}}];
		\draw (6.5,-5.5) --(5.5,-6.5)[postaction={decorate, decoration={markings,mark=at position .5 with {\arrow[black]{stealth}}}}];
		\draw  (2.5,-6.5) --  (3,-7.5)[postaction={decorate, decoration={markings,mark=at position .5 with {\arrow[black]{stealth}}}}];
		\draw (5.5,-6.5) -- (5,-7.5)[postaction={decorate, decoration={markings,mark=at position .5 with {\arrow[black]{stealth}}}}];
		\draw  (9,-5.5)-- (9.5,-7.5)[postaction={decorate, decoration={markings,mark=at position .5 with {\arrow[black]{stealth}}}}];
		\draw  (10,-5.5) -- (10,-6.5)[postaction={decorate, decoration={markings,mark=at position .5 with {\arrow[black]{stealth}}}}];
		\draw  (11,-5.5) -- (10.5,-7.5)[postaction={decorate, decoration={markings,mark=at position .5 with {\arrow[black]{stealth}}}}];
		\draw (13,-5.5) -- (13,-6.5)[postaction={decorate, decoration={markings,mark=at position .5 with {\arrow[black]{stealth}}}}];
		\draw[fill] (v63) circle [radius=0.1];
		\draw[fill] (v67) circle [radius=0.1];
		\draw[fill] (v75) circle [radius=0.1];
		\draw[fill] (v79) circle [radius=0.1];
		\node (v80) at (3,-8) {};
		\node (v82) at (5,-8) {};
		\node (v83) at (9.5,-8) {};
		\node (v85) at (10.5,-8) {};
		\node (v81) at (4,-10) {};
		\node (v84) at (10,-10) {};
		\draw  (3,-8)-- (4,-10)[postaction={decorate, decoration={markings,mark=at position .5 with {\arrow[black]{stealth}}}}];
		\draw  (5,-8) -- (4,-10)[postaction={decorate, decoration={markings,mark=at position .5 with {\arrow[black]{stealth}}}}];
		\draw  (9.5,-8) -- (10,-10)[postaction={decorate, decoration={markings,mark=at position .5 with {\arrow[black]{stealth}}}}];
		\draw (10.5,-8) -- (10,-10)[postaction={decorate, decoration={markings,mark=at position .5 with {\arrow[black]{stealth}}}}];
		\draw[fill] (v81) circle [radius=0.1];
		\draw[fill] (v84) circle [radius=0.1];
			\end{tikzpicture}

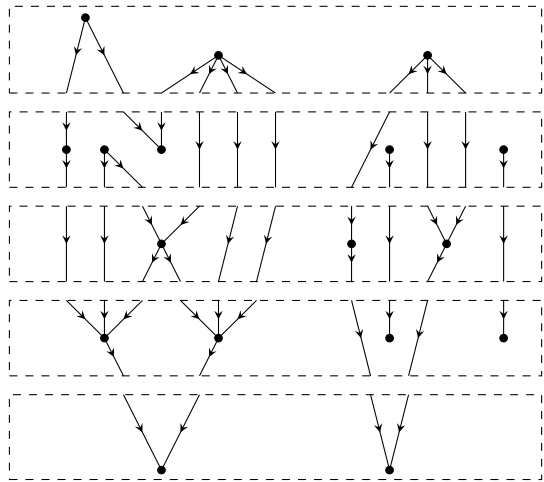
 Figure $5$. Make each layer into an elementary progressive plane graph
\end{center}
	\begin{center}
		\begin{tikzpicture}[scale=0.5]
			
			\draw [dashed] (0,2.3) rectangle (14,0);
			\draw [dashed] (0,-0.5) rectangle (14,-2.5);
			\draw [dashed] (0,-3) rectangle (14,-5);
			\draw [dashed] (0,-5.5) rectangle (14,-7.5);
			\draw [dashed] (0,-8) rectangle (14,-10.25);
			
			\draw  (2,2) rectangle (2,2) node (v1) {};
			\draw  (1.5,0) rectangle (1.5,0) node (v2) {};
			\draw  (3,0) rectangle (3,0) node (v3) {};
			\draw  (4,0) rectangle (4,0) node (v5) {};
			\draw  (5,0) rectangle (5,0) node (v6) {};
			\draw  (6,0) rectangle (6,0) node (v7) {};
			\draw  (7,0) rectangle (7,0) node (v8) {};
			\draw  (10,0) rectangle (10,0) node (v10) {};
			\draw  (11,0) rectangle (11,0) node (v11) {};
			\draw  (12,0) rectangle (12,0) node (v12) {};
			\draw  (5.5,1) rectangle (5.5,1) node (v4) {};
			\draw  (11,1) rectangle (11,1) node (v9) {};
			
			\draw  (2,2)-- (1.5,0)[postaction={decorate, decoration={markings,mark=at position .5 with {\arrow[black]{stealth}}}}];
			\draw  (2,2)-- (3,0)[postaction={decorate, decoration={markings,mark=at position .5 with {\arrow[black]{stealth}}}}];
			\draw  (5.5,1) -- (4,0)[postaction={decorate, decoration={markings,mark=at position .5 with {\arrow[black]{stealth}}}}];
			\draw  (5.5,1) -- (5,0)[postaction={decorate, decoration={markings,mark=at position .5 with {\arrow[black]{stealth}}}}];
			\draw(5.5,1) -- (6,0)[postaction={decorate, decoration={markings,mark=at position .5 with {\arrow[black]{stealth}}}}];
			\draw (5.5,1) -- (7,0)[postaction={decorate, decoration={markings,mark=at position .5 with {\arrow[black]{stealth}}}}];
			\draw (11,1) -- (10,0)[postaction={decorate, decoration={markings,mark=at position .5 with {\arrow[black]{stealth}}}}];
			\draw  (11,1)-- (11,0)[postaction={decorate, decoration={markings,mark=at position .5 with {\arrow[black]{stealth}}}}];
			\draw  (11,1)-- (12,0)[postaction={decorate, decoration={markings,mark=at position .5 with {\arrow[black]{stealth}}}}];
			\draw[fill] (v1) circle [radius=0.1];
			\draw[fill] (v4) circle [radius=0.1];
			\draw[fill] (v9) circle [radius=0.1];

			\node (v29) at (9,-2.5) {};
			\node (v31) at (10,-2.5) {};
			\node (v33) at (11,-2.5) {};
			\node (v35) at (12,-2.5) {};
			\node (v37) at (13,-2.5) {};
			\node (v14) at (1.5,-1.5) {};
			\node (v16) at (2.5,-1.5) {};
			\node (v20) at (4,-1.5) {};
			\node (v30) at (10,-1.5) {};
			\node (v36) at (13,-1.5) {};
			\node (v13) at (1.5,-0.5) {};
			\node (v19) at (3,-0.5) {};
			\node (v21) at (4,-0.5) {};
			\node (v22) at (5,-0.5) {};
			\node (v24) at (6,-0.5) {};
			\node (v26) at (7,-0.5) {};
			\node (v28) at (10,-0.5) {};
			\node (v32) at (11,-0.5) {};
			\node (v34) at (12,-0.5) {};
			\node (v15) at (1.5,-2.5) {};
			\node (v17) at (2.5,-2.5) {};
			\node (v18) at (3.5,-2.5) {};
			\node (v23) at (5,-2.5) {};
			\node (v25) at (6,-2.5) {};
			\node (v27) at (7,-2.5) {};
			\draw (1.5,-0.5) --(1.5,-1.5)[postaction={decorate, decoration={markings,mark=at position .5 with {\arrow[black]{stealth}}}}];
			\draw  (1.5,-1.5) -- (1.5,-2.5)[postaction={decorate, decoration={markings,mark=at position .5 with {\arrow[black]{stealth}}}}];
			\draw  (2.5,-1.5) --  (2.5,-2.5)[postaction={decorate, decoration={markings,mark=at position .5 with {\arrow[black]{stealth}}}}];
			\draw (2.5,-1.5) -- (3.5,-2.5)[postaction={decorate, decoration={markings,mark=at position .5 with {\arrow[black]{stealth}}}}];
			\draw  (3,-0.5)--  (4,-1.5)[postaction={decorate, decoration={markings,mark=at position .5 with {\arrow[black]{stealth}}}}];
			\draw  (4,-0.5)--  (4,-1.5)[postaction={decorate, decoration={markings,mark=at position .5 with {\arrow[black]{stealth}}}}];
			\draw (5,-0.5)--  (5,-2.5)[postaction={decorate, decoration={markings,mark=at position .5 with {\arrow[black]{stealth}}}}];
			\draw  (6,-0.5)-- (6,-2.5)[postaction={decorate, decoration={markings,mark=at position .5 with {\arrow[black]{stealth}}}}];
			\draw  (7,-0.5) -- (7,-2.5)[postaction={decorate, decoration={markings,mark=at position .5 with {\arrow[black]{stealth}}}}];
			\draw  (10,-0.5) -- (9,-2.5)[postaction={decorate, decoration={markings,mark=at position .5 with {\arrow[black]{stealth}}}}];
			\draw   (10,-1.5)-- (10,-2.5)[postaction={decorate, decoration={markings,mark=at position .5 with {\arrow[black]{stealth}}}}];
			\draw (11,-0.5)  -- (11,-2.5)[postaction={decorate, decoration={markings,mark=at position .5 with {\arrow[black]{stealth}}}}];
			\draw (12,-0.5)-- (12,-2.5)[postaction={decorate, decoration={markings,mark=at position .5 with {\arrow[black]{stealth}}}}];
			\draw  (13,-1.5) -- (13,-2.5)[postaction={decorate, decoration={markings,mark=at position .5 with {\arrow[black]{stealth}}}}];
			\draw[fill] (v14) circle [radius=0.1];
			\draw[fill] (v16) circle [radius=0.1];
			\draw[fill] (v20) circle [radius=0.1];
			\draw[fill] (v36) circle [radius=0.1];
			\draw[fill] (v30) circle [radius=0.1];
			\node (v38) at (1.5,-3) {};
			\node (v40) at (2.5,-3) {};
			\node (v42) at (3.5,-3) {};
			\node (v44) at (5,-3) {};
			\node (v47) at (6,-3) {};
			\node (v49) at (7,-3) {};
			\node (v51) at (9,-3) {};
			\node (v54) at (10,-3) {};
			\node (v56) at (11,-3) {};
			\node (v58) at (12,-3) {};
			\node (v60) at (13,-3) {};
			\node (v39) at (1.5,-5) {};
			\node (v41) at (2.5,-5) {};
			\node (v45) at (3.5,-5) {};
			\node (v46) at (4.5,-5) {};
			\node (v48) at (5.5,-5) {};
			\node (v50) at (6.5,-5) {};
			\node (v53) at (9,-5) {};
			\node (v55) at (10,-5) {};
			\node (v59) at (11,-5) {};
			\node (v61) at (13,-5) {};
			\draw  (1.5,-3) -- (1.5,-5)[postaction={decorate, decoration={markings,mark=at position .5 with {\arrow[black]{stealth}}}}];
			\draw (2.5,-3) -- (2.5,-5)[postaction={decorate, decoration={markings,mark=at position .5 with {\arrow[black]{stealth}}}}];
			\node (v43) at (4,-4) {};
			\draw  (3.5,-3) -- (4,-4)[postaction={decorate, decoration={markings,mark=at position .5 with {\arrow[black]{stealth}}}}];
			\draw  (5,-3)-- (4,-4)[postaction={decorate, decoration={markings,mark=at position .5 with {\arrow[black]{stealth}}}}];
			\draw  (4,-4)--  (3.5,-5)[postaction={decorate, decoration={markings,mark=at position .5 with {\arrow[black]{stealth}}}}];
			\draw  (4,-4) -- (4.5,-5)[postaction={decorate, decoration={markings,mark=at position .5 with {\arrow[black]{stealth}}}}];
			\draw  (6,-3) --  (5.5,-5)[postaction={decorate, decoration={markings,mark=at position .5 with {\arrow[black]{stealth}}}}];
			\draw (7,-3)-- (6.5,-5)[postaction={decorate, decoration={markings,mark=at position .5 with {\arrow[black]{stealth}}}}];
			\node (v52) at (9,-4) {};
			\draw  (9,-3) --(9,-4)[postaction={decorate, decoration={markings,mark=at position .5 with {\arrow[black]{stealth}}}}];
			\draw  (9,-4) -- (9,-5)[postaction={decorate, decoration={markings,mark=at position .5 with {\arrow[black]{stealth}}}}];
			\draw(10,-3) -- (10,-5)[postaction={decorate, decoration={markings,mark=at position .5 with {\arrow[black]{stealth}}}}];
			\node (v57) at (11.5,-4) {};
			\draw  (11,-3)--  (11.5,-4)[postaction={decorate, decoration={markings,mark=at position .5 with {\arrow[black]{stealth}}}}];
			\draw  (12,-3)--  (11.5,-4)[postaction={decorate, decoration={markings,mark=at position .5 with {\arrow[black]{stealth}}}}];
			\draw   (11.5,-4)-- (11,-5)[postaction={decorate, decoration={markings,mark=at position .5 with {\arrow[black]{stealth}}}}];
			\draw (13,-3)-- (13,-5)[postaction={decorate, decoration={markings,mark=at position .5 with {\arrow[black]{stealth}}}}];
			\draw[fill] (v43) circle [radius=0.1];
			\draw[fill] (v52) circle [radius=0.1];
			\draw[fill] (v57) circle [radius=0.1];
			\node (v62) at (1.5,-5.5) {};
			\node (v64) at (2.5,-5.5) {};
			\node (v65) at (3.5,-5.5) {};
			\node (v66) at (4.5,-5.5) {};
			\node (v68) at (5.5,-5.5) {};
			\node (v69) at (6.5,-5.5) {};
			\node (v72) at (9,-5.5) {};
			\node (v74) at (10,-5.5) {};
			\node (v76) at (11,-5.5) {};
			\node (v78) at (13,-5.5) {};
			\node (v70) at (3,-7.5) {};
			\node (v71) at (5,-7.5) {};
			\node (v73) at (9.5,-7.5) {};
			\node (v77) at (10.5,-7.5) {};
			\node (v63) at (2.5,-6.5) {};
			\node (v67) at (5.5,-6.5) {};
			\node (v75) at (10,-6.5) {};
			\node (v79) at (13,-6.5) {};
			\draw  (1.5,-5.5)-- (2.5,-6.5)[postaction={decorate, decoration={markings,mark=at position .5 with {\arrow[black]{stealth}}}}];
			\draw (2.5,-5.5) -- (2.5,-6.5)[postaction={decorate, decoration={markings,mark=at position .5 with {\arrow[black]{stealth}}}}];
			\draw  (3.5,-5.5) --(2.5,-6.5)[postaction={decorate, decoration={markings,mark=at position .5 with {\arrow[black]{stealth}}}}];
			\draw  (4.5,-5.5) -- (5.5,-6.5)[postaction={decorate, decoration={markings,mark=at position .5 with {\arrow[black]{stealth}}}}];
			\draw  (5.5,-5.5) -- (5.5,-6.5)[postaction={decorate, decoration={markings,mark=at position .5 with {\arrow[black]{stealth}}}}];
			\draw (6.5,-5.5) --(5.5,-6.5)[postaction={decorate, decoration={markings,mark=at position .5 with {\arrow[black]{stealth}}}}];
			\draw  (2.5,-6.5) --  (3,-7.5)[postaction={decorate, decoration={markings,mark=at position .5 with {\arrow[black]{stealth}}}}];
			\draw (5.5,-6.5) -- (5,-7.5)[postaction={decorate, decoration={markings,mark=at position .5 with {\arrow[black]{stealth}}}}];
			\draw  (9,-5.5)-- (9.5,-7.5)[postaction={decorate, decoration={markings,mark=at position .5 with {\arrow[black]{stealth}}}}];
			\draw  (10,-5.5) -- (10,-6.5)[postaction={decorate, decoration={markings,mark=at position .5 with {\arrow[black]{stealth}}}}];
			\draw  (11,-5.5) -- (10.5,-7.5)[postaction={decorate, decoration={markings,mark=at position .5 with {\arrow[black]{stealth}}}}];
			\draw (13,-5.5) -- (13,-6.5)[postaction={decorate, decoration={markings,mark=at position .5 with {\arrow[black]{stealth}}}}];
			\draw[fill] (v63) circle [radius=0.1];
			\draw[fill] (v67) circle [radius=0.1];
			\draw[fill] (v75) circle [radius=0.1];
			\draw[fill] (v79) circle [radius=0.1];
			\node (v80) at (3,-8) {};
			\node (v82) at (5,-8) {};
			\node (v83) at (9.5,-8) {};
			\node (v85) at (10.5,-8) {};
			\node (v81) at (4,-10) {};
			\node (v84) at (10,-10) {};
			\draw  (3,-8)-- (4,-10)[postaction={decorate, decoration={markings,mark=at position .5 with {\arrow[black]{stealth}}}}];
			\draw  (5,-8) -- (4,-10)[postaction={decorate, decoration={markings,mark=at position .5 with {\arrow[black]{stealth}}}}];
			\draw  (9.5,-8) -- (10,-10)[postaction={decorate, decoration={markings,mark=at position .5 with {\arrow[black]{stealth}}}}];
			\draw (10.5,-8) -- (10,-10)[postaction={decorate, decoration={markings,mark=at position .5 with {\arrow[black]{stealth}}}}];
			\draw[fill] (v81) circle [radius=0.1];
			\draw[fill] (v84) circle [radius=0.1];
			       \node[scale=0.5] at (1.5,0.8) {$1$};
					\node[scale=0.5] at (2.5,0.6) {$2$};
					\node[scale=0.5] at (4.3,0.4) {$3$};
					\node [scale=0.5]at (4.9,0.2) {$4$};
					\node [scale=0.5]at (5.5,0.3) {$5$};
					\node [scale=0.5]at (6.6,0.5) {$6$};
					\node [scale=0.5]at (10,0.3) {$7$};
					\node[scale=0.5] at (10.8,0.2) {$8$};
					\node[scale=0.5] at (11.8,0.5) {$9$};

					\node[scale=0.5] at (1.3,-0.9) {$1$};
					\node[scale=0.5] at (1.3,-2.1) {$2$};
					\node [scale=0.5]at (2.3,-2.2) {$3$};
					\node[scale=0.5] at (3.3,-1.9) {$4$};
					\node[scale=0.5] at (3.2,-1) {$5$};
					\node[scale=0.5] at (4.2,-0.9) {$6$};
					\node[scale=0.5] at (4.8,-1.5) {$7$};
					\node [scale=0.5]at (5.8,-1.5) {$8$};
					\node[scale=0.5] at (6.8,-1.5) {$9$};
					\node[scale=0.5] at (9,-1.9) {$10$};
					\node [scale=0.5]at (9.7,-2.2) {$11$};
					\node[scale=0.5] at (10.7,-1.7) {$12$};
					\node[scale=0.5] at (11.8,-1.7) {$13$};
					\node[scale=0.5] at (12.8,-2.1) {$14$};
					
					\node[scale=0.5] at (1.2,-4.1) {$1$};
					\node [scale=0.5]at (2.2,-4) {$2$};
					\node [scale=0.5]at (3.4,-3.6) {$3$};
					\node[scale=0.5] at (4.8,-3.6) {$4$};
					\node [scale=0.5]at (3.4,-4.7) {$5$};
					\node[scale=0.5] at (4.6,-4.6) {$6$};
					\node[scale=0.5] at (5.5,-4.1) {$7$};
					\node [scale=0.5]at (6.6,-4.1) {$8$};
					\node [scale=0.5]at (8.7,-3.5) {$9$};
					\node [scale=0.5]at (8.7,-4.6) {$10$};
					\node [scale=0.5]at (9.8,-4.1) {$11$};
					\node [scale=0.5]at (10.9,-3.5) {$12$};
					\node [scale=0.5]at (12.1,-3.5) {$13$};
					\node [scale=0.5]at (10.9,-4.5) {$14$};
					\node [scale=0.5]at (13.2,-4.1) {$15$};

					\node [scale=0.5] at (1.7,-5.9) {$1$};
					\node [scale=0.5] at (2.3,-5.8) {$2$};
					\node [scale=0.5] at (2.9,-5.8) {$3$};
					\node  [scale=0.5]at (2.6,-7.2) {$4$};
					\node  [scale=0.5]at (4.6,-6.1) {$5$};
					\node  [scale=0.5]at (5.3,-5.9) {$6$};
					\node  [scale=0.5]at (6,-5.7) {$7$};
					\node  [scale=0.5]at (5.6,-7.1) {$8$};
					\node  [scale=0.5]at (8.9,-6.2) {$9$};
					\node [scale=0.5] at (9.6,-6) {$10$};
					\node  [scale=0.5]at (10.6,-6) {$11$};
					\node  [scale=0.5]at (12.6,-6) {$12$};
					
					\node[scale=0.5] at (3.1,-9.1) {$1$};
					\node[scale=0.5] at (4.9,-9.1) {$2$};
					\node [scale=0.5]at (9.4,-9) {$3$};
					\node[scale=0.5] at (10.6,-8.8) {$4$};
		\end{tikzpicture}

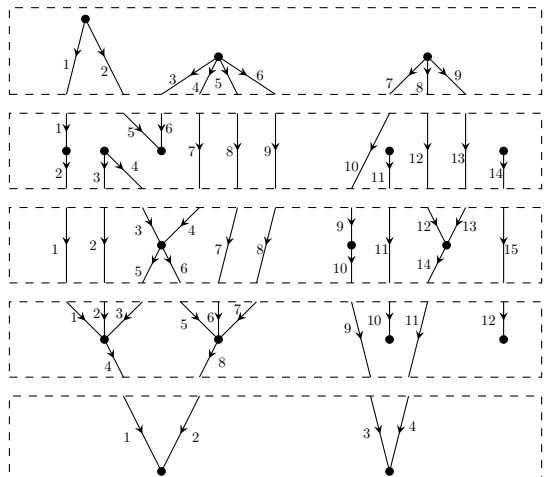
 Figure $6$. Order edges linearly from left to right and up to down
\end{center}

The crux of this strategy is to define the composition rule and show that the composition of two upward planar orders is still an upward planar order. Similar to the composition theory for planar orders \cite{[HLY19]}, the composition rule is also of shuffle style and associative. The main result of this paper is Theorem \ref{main}.

The paper is organized as follows. In Section $2$, we recall the notion of an upward planar order. In Section $3$, we recall the notion of a progressive graph and introduce the notion of admissible condition for upward planar orders on a progressive graph. In Section $4$, we introduce the composition rule of upward planar orders. Section $5$ is devoted to the proof of the main result.

\section{Upward planar order}
We start by introducing some notations. For a finite linearly ordered set $(S,\leq)$, and any subset $X\subseteq S$, we denote the minimal element $min\ X$ and the maximal element $max\ X$ of $X$ by $X^-$ and $X^+$, respectively. The convex hull of $X$ in $S$ is the set $\overline{X}=\{y\in S|X^-\leq y\leq X^+\}$, which is the smallest closed interval containing $X$.

For an acyclic directed graph $G$, the edge set and vertex set are  denoted by $E(G)$ and $V(G)$, respectively. The notation $e_1\to e_2$ means that there is a directed path starting from edge $e_1$ and ending with edge $e_2$, which, by the acyclicity of $G$, defines a partial order $\to$ on $E(G)$ and is called the \textbf{reachable order} of $G$. A vertex $v$ is called \textbf{processive} \cite{[HLY19]} if it is neither a source nor a sink, or in other words, both $I(v)$ and $O(v)$ are nonempty.

The notion of an upward planar order was introduced in \cite{[LY19]} to combinatorially characterize an upward plane graph.
\begin{defn}\label{upo}
An \textbf{upward planar order} on an acyclic directed graph $G$ is a linear order $\preceq$ on $E(G)$, such that

$(U_1)$ for edges $e_1, e_2$, $e_1\rightarrow e_2$ implies that $e_1\prec e_2$;

$(U_2)$ for any vertex $v$, $\overline{I(v)}\cap \overline{O(v)}=\emptyset$ and $\overline{E(v)}=\overline{I(v)}\sqcup \overline{O(v)}$;

$(U_3)$ for any two  vertices $v_1$ and $v_2$, $I(v_1)\cap \overline{I(v_2)}\neq\emptyset$ implies that $\overline{I(v_1)}\subseteq \overline{I(v_2)}$, and $O(v_1)\cap \overline{O(v_2)}\neq\emptyset$ implies that $\overline{O(v_1)}\subseteq \overline{O(v_2)}$.
\end{defn}
The $(U_1)$ condition says that $\preceq$ is a linear extension of the reachable order $\to$. The $(U_2)$ condition, together with $(U_1)$, implies that for any processive vertex $v$, $O(v)^-=I(v)^++1$, which means that the minmal outgoing edge $O(v)^-$  is just right next to the maximal incoming edge $I(v)^+$ in the linear list of $E(G)$ with respect to $\preceq$. For a source or a sink, $(U_2)$ is unconditionally satisfied. Notice that for any two subsets $X,Y$ of a linearly ordered set, $X$ $\subseteq$ $\overline{Y}$ is equivalent to $\overline{X}$  $\subseteq$ $\overline{Y}$, thus the condition $(U_3$) can be restated as follows. For any two vertices $v_1$, $v_2$, $I(v_1)\cap \overline{I(v_2)}\neq\emptyset$ implies that $I(v_1)\subseteq \overline{I(v_2)}$, and $O(v_1)\cap \overline{O(v_2)}\neq\emptyset$ implies that $O(v_1)\subseteq \overline{O(v_2)}$.

\begin{ex}\label{type}
The following is a typical example of progressive plane graph that motivates the Definition \ref{upo}, where  the vertices drawn on the box boundary are omitted for convenience. The natural convention of edge-ordering is first from left to right and then from up to down.
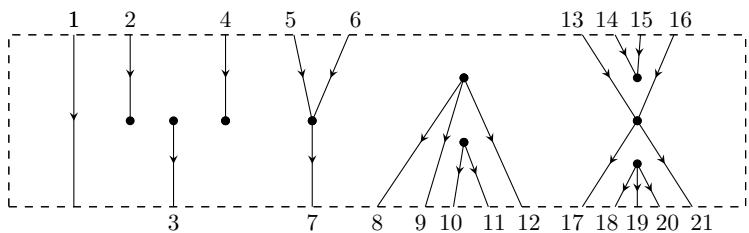
\begin{center}
	\begin{tikzpicture}[scale=0.57]
		
		\draw[dashed]  (-4,2.5) rectangle (13,-1.5);
		\node [scale=0.8, above] (v1) at (-2.5,2.5) {$1$};
		\node [scale=0.8,below] (v2) at (-2.5,-1.5) {};
		\node (v3) at (1,0.5) {};
		\draw [fill](v3) circle [radius=0.09];
		\node [scale=0.8,below](v4) at (-0.2,-1.5) {$3$};
		\node [scale=0.8,above](v5) at (2.5,2.5) {$5$};
		\node[scale=0.8,above] (v7) at (4,2.5) {$6$};
		\node (v6) at (3,0.5) {};
		\draw [fill](v6) circle [radius=0.09];
		\node [scale=0.8,below](v8) at (3,-1.5) {$7$};
		\node (v9) at (6.5,1.5) {};
		\draw [fill](v9) circle [radius=0.09];
		\node [scale=0.8,below](v10) at (4.5,-1.5) {$8$};
		\node [scale=0.8,below](v11) at (5.5,-1.5) {$9$};
		\node [scale=0.8,below](v13) at (6.2,-1.5) {$10$};
		\node[scale=0.8,below] (v14) at (7.2,-1.5) {$11$};
		\node[scale=0.8,below] (v15) at (8,-1.5) {$12$};
		\node [scale=0.8,above](v16) at (9,2.5) {$13$};
		\node [scale=0.8,above](v23) at (10.6,2.5) {$15$};
		\node [scale=0.8,above](v18) at (11.5,2.5) {$16$};
		\node (v17) at (10.5,0.5) {};
		\draw [fill](v17) circle [radius=0.09];
		\node (v12) at (6.5,0) {};
		\draw [fill](v12) circle [radius=0.09];
		
		\draw  (v1) -- (v2)[postaction={decorate, decoration={markings,mark=at position .5 with {\arrow[black]{stealth}}}}];
		\draw  (-0.2,0.5) -- (v4)[postaction={decorate, decoration={markings,mark=at position .5 with {\arrow[black]{stealth}}}}];
		\draw  (v5) -- (3,0.5)[postaction={decorate, decoration={markings,mark=at position .5 with {\arrow[black]{stealth}}}}];
		\draw  (v7) -- (3,0.5)[postaction={decorate, decoration={markings,mark=at position .5 with {\arrow[black]{stealth}}}}];
		\draw  (3,0.5)-- (v8)[postaction={decorate, decoration={markings,mark=at position .5 with {\arrow[black]{stealth}}}}];
		\draw  (6.5,1.5)-- (4.5,-1.5)[postaction={decorate, decoration={markings,mark=at position .5 with {\arrow[black]{stealth}}}}];
		\draw  (6.5,1.5) -- (v11)[postaction={decorate, decoration={markings,mark=at position .5 with {\arrow[black]{stealth}}}}];
		\draw (6.5,0) -- (v13)[postaction={decorate, decoration={markings,mark=at position .5 with {\arrow[black]{stealth}}}}];
		\draw (6.5,0) -- (v14)[postaction={decorate, decoration={markings,mark=at position .5 with {\arrow[black]{stealth}}}}];
		\draw  (6.5,1.5) -- (v15)[postaction={decorate, decoration={markings,mark=at position .5 with {\arrow[black]{stealth}}}}];
		\node [scale=0.8,below](v19) at (9,-1.5) {$17$};
		\node (v22) at (10.5,1.5) {};
		\draw [fill](v22) circle [radius=0.09];
		\node [scale=0.8,above](v21) at (9.8,2.5) {$14$};
		\draw  (v16) -- (10.5,0.5)[postaction={decorate, decoration={markings,mark=at position .5 with {\arrow[black]{stealth}}}}];
		\draw  (v18) -- (10.5,0.5)[postaction={decorate, decoration={markings,mark=at position .5 with {\arrow[black]{stealth}}}}];
		\draw  ((10.5,0.5) -- (v19)[postaction={decorate, decoration={markings,mark=at position .5 with {\arrow[black]{stealth}}}}];
		\node [scale=0.8,below](v20) at (12,-1.5) {$21$};
		\draw (10.5,0.5) --(v20)[postaction={decorate, decoration={markings,mark=at position .5 with {\arrow[black]{stealth}}}}];
		\draw  (v21) -- (10.5,1.5)[postaction={decorate, decoration={markings,mark=at position .5 with {\arrow[black]{stealth}}}}];
		\draw  (v23) -- (10.5,1.5)[postaction={decorate, decoration={markings,mark=at position .5 with {\arrow[black]{stealth}}}}];
		\node (v24) at (10.5,-0.5) {};
		\draw [fill](v24) circle [radius=0.09];
		\node[scale=0.8,below] (v25) at (9.8,-1.5) {$18$};
		\node [scale=0.8,below](v26) at (10.5,-1.5) {$19$};
		\node [scale=0.8,below](v27) at (11.2,-1.5) {$20$};
		\draw (10.5,-0.5) -- (v25)[postaction={decorate, decoration={markings,mark=at position .65 with {\arrow[black]{stealth}}}}];
		\draw  (10.5,-0.5) -- (v26)[postaction={decorate, decoration={markings,mark=at position .65 with {\arrow[black]{stealth}}}}];
		\draw  (10.5,-0.5) -- (v27)[postaction={decorate, decoration={markings,mark=at position .65 with {\arrow[black]{stealth}}}}];
		
		\draw[dashed]  (-4,2.5) rectangle (13,-1.5);
		\node [scale=0.8, above] (v1) at (-2.5,2.5) {$1$};
		\node [scale=0.8,below] (v2) at (-2.5,-1.5) {};
		\node (v3) at (-0.2,0.5) {};
		\draw [fill](v3) circle [radius=0.09];
		\node [scale=0.8,below](v4) at (1,-1.5) {};
		\node [scale=0.8,above](v5) at (2.5,2.5) {};
		\node[scale=0.8,above] (v7) at (4,2.5) {};
		\node (v6) at (3,0.5) {};
		\draw [fill](v6) circle [radius=0.09];
		\node [scale=0.8,below](v8) at (3,-1.5) {};
		\node (v9) at (6.5,1.5) {};
		\draw [fill](v9) circle [radius=0.09];
		
		\node (v17) at (10.5,0.5) {};
		\draw [fill](v17) circle [radius=0.09];
		\node (v12) at (6.5,0) {};
		\draw [fill](v12) circle [radius=0.09];
		
		\node [scale=0.8,above](v40) at (-1.2,2.5) {$2$};
		\node (v41) at (-1.2,0.5) {};
		\draw [fill](v41) circle [radius=0.09];
		\draw (-1.2,2.5) -- (-1.2,0.5)[postaction={decorate, decoration={markings,mark=at position .5 with {\arrow[black]{stealth}}}}];
		
		\node [scale=0.8,above](v50) at (1,2.5) {$4$};
		\node (v51) at (1,0.5) {};
		\draw [fill](v51) circle [radius=0.09];
		\draw (1,2.5) -- (1,0.5)[postaction={decorate, decoration={markings,mark=at position .5 with {\arrow[black]{stealth}}}}];
	\end{tikzpicture}
\end{center}
\end{ex}

The notion of an upward planar order is a natural generalization of that of a planar order \cite{[HLY19]}, which was introduced  to characterize an equivalence class of \textbf{processive} plane graphs (or plane $st$ graphs). An acyclic directed graph together with an upward planar order is called an \textbf{upward planarly ordered graph}, or \textbf{UPO-graph} for short. A UPO-graph is a combinatorial version of an upward plane graph.

 The following is a new definition of upward planar order, whose form is much more similar to that of planar order.
\begin{defn}\label{upo1}
An \textbf{upward planar order} on acyclic directed graph $G$ is a linear order $\prec$ on the edge set $E(G)$, such that

$(Q_1)$ $e_1\rightarrow e_2$ implies that $e_1\prec e_2$;

$(Q_2)$ if $e_1\prec e\prec e_2$ and $e_1$, $e_2$ are adjacent, then
\begin{equation*}
\begin{cases}
I(t(e))\subseteq \overline{I(v)},&  \text{if}\ t(e_1)=t(e_2)=v;\\
O(s(e))\subseteq \overline{O(v)},&  \text{if}\ s(e_1)=s(e_2)=v;\\
\text{either}\ I(t(e))\subseteq \overline{I(v)}\  \text{or}\ O(s(e))\subseteq \overline{O(v)},&  \text{if}\ t(e_1)=s(e_2)=v.
\end{cases}
\end{equation*}
\end{defn}

The four possible configurations of $e_1\prec e\prec e_2$ with $e_1, e_2$ adjacent are shown as follows.
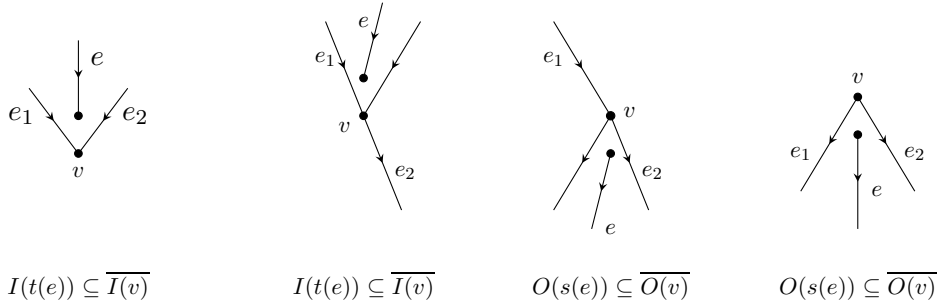
\begin{center}
\begin{tikzpicture}[scale=0.5]
\node (v2) at (-4.5,1.5) {};
\node (v1) at (-6,3.5) {};
\node (v3) at (-3,3.5) {};
\draw  (v1) -- (-4.5,1.5)[postaction={decorate, decoration={markings,mark=at position .5 with {\arrow[black]{stealth}}}}];
\draw  (v3) --(-4.5,1.5)[postaction={decorate, decoration={markings,mark=at position .5 with {\arrow[black]{stealth}}}}];
\node (v5) at (-4.5,2.5) {};
\node (v4) at (-4.5,4.5) {};
\draw (-4.5,4.5)-- (-4.5,2.5)[postaction={decorate, decoration={markings,mark=at position .5 with {\arrow[black]{stealth}}}}];
\draw[fill] (v2) circle [radius=0.1];
\draw[fill] (v5) circle [radius=0.1];
\node at (-6,2.5) {$e_1$};
\node at (-3,2.5) {$e_2$};
\node at (-4,4) {$e$};
\node (v7) at (3,2.5) {};
\node (v6) at (2,5) {};
\node (v8) at (4,0) {};
\draw  (2,5) --(3,2.5)[postaction={decorate, decoration={markings,mark=at position .5 with {\arrow[black]{stealth}}}}];
\draw  (3,2.5)-- (4,0)[postaction={decorate, decoration={markings,mark=at position .5 with {\arrow[black]{stealth}}}}];
\node (v9) at (4.5,5) {};
\draw  (4.5,5)-- (3,2.5)[postaction={decorate, decoration={markings,mark=at position .5 with {\arrow[black]{stealth}}}}];
\draw[fill] (v7) circle [radius=0.1];
\node (v11) at (3,3.5) {};
\node (v10) at (3.5,5.5) {};
\draw  (3.5,5.5) --(3,3.5)[postaction={decorate, decoration={markings,mark=at position .5 with {\arrow[black]{stealth}}}}];
\draw[fill] (v11) circle [radius=0.1];
\node (v13) at (9.5,2.5) {};
\node (v12) at (8,5) {};
\node (v14) at (10.5,0) {};
\draw  (8,5) -- (9.5,2.5)[postaction={decorate, decoration={markings,mark=at position .5 with {\arrow[black]{stealth}}}}];
\draw  (9.5,2.5) -- (10.5,0)[postaction={decorate, decoration={markings,mark=at position .5 with {\arrow[black]{stealth}}}}];
\draw[fill] (v13) circle [radius=0.11];
\node (v15) at (8,0) {};
\draw  (9.5,2.5) -- (8,0)[postaction={decorate, decoration={markings,mark=at position .5 with {\arrow[black]{stealth}}}}];
\node (v16) at (9.5,1.5) {};
\node (v17) at (9,-1) {};
\draw  (9.5,1.5)-- (9,-0.5)[postaction={decorate, decoration={markings,mark=at position .5 with {\arrow[black]{stealth}}}}];
\draw[fill] (v16) circle [radius=0.1];
\node (v18) at (16,3) {};
\node (v19) at (14.5,0.5) {};
\node (v20) at (17.5,0.5) {};
\node (v21) at (16,2) {};
\node (v22) at (16,-1) {};
\draw  (16,3) -- (14.5,0.5)[postaction={decorate, decoration={markings,mark=at position .5 with {\arrow[black]{stealth}}}}];
\draw (16,3) --(17.5,0.5)[postaction={decorate, decoration={markings,mark=at position .5 with {\arrow[black]{stealth}}}}];
\draw  (16,2) -- (16,-0.5)[postaction={decorate, decoration={markings,mark=at position .5 with {\arrow[black]{stealth}}}}];
\draw[fill] (v21) circle [radius=0.1];
\draw[fill] (v18) circle [radius=0.1];
\node[scale=0.8] at (2,4) {$e_1$};
\node[scale=0.8] at (4.1,1) {$e_2$};
\node[scale=0.8] at (3,5) {$e$};
\node[scale=0.8] at (8,4) {$e_1$};
\node[scale=0.8] at (10.6,1) {$e_2$};
\node[scale=0.8] at (9.5,-0.5) {$e$};
\node[scale=0.8] at (14.5,1.5) {$e_1$};
\node[scale=0.8] at (17.5,1.5) {$e_2$};
\node[scale=0.8] at (16.5,0.5) {$e$};
\node[scale=0.8] at (2.5,2.2) {$v$};
\node [scale=0.8]at (10,2.6) {$v$};
\node[scale=0.8] at (16,3.5) {$v$};
\node [scale=0.8]at (-4.5,1) {$v$};
\node[scale=0.8] at (-4.5,-2) {$I(t(e))\subseteq \overline{I(v)}$};
\node [scale=0.8]at (3,-2) {$I(t(e))\subseteq \overline{I(v)}$};
\node [scale=0.8]at (9.5,-2) {$O(s(e))\subseteq \overline{O(v)}$};
\node [scale=0.8]at (16,-2) {$O(s(e))\subseteq \overline{O(v)}$};
\end{tikzpicture}
\end{center}

$(Q_1)=(U_1)$ is the linear extension condition. $(Q_2)$ is called  the \textbf{nesting condition}, which is an equivalent formulation of $(U_2)$ and $(U_3)$.

\begin{thm}\label{eq}
Definition \ref{upo} and Definition \ref{upo1} are equivalent.
\end{thm}
\begin{proof}
Clearly, $(U_1)=(Q_1)$. We only need to show that $(U_2)+(U_3)\Longleftrightarrow (Q_2)$.

$(\Longleftarrow)$. To show $(Q_2)\Longrightarrow(U_2)$, we only need to show that for any processive vertex $v$, $O(v)^-=I(v)^++1$. We prove this by contradiction. Suppose there exist a processive vertex $v$ and an edge $e$ such that $I(v)^+\prec e \prec O(v)^-$. Clearly, $v=t(I(v)^+)=s(O(v)^-)$, then by $(Q_2)$, we have either $e\in I(t(e))\subseteq \overline{I(v)}$ or $e\in O(s(e))\subseteq\overline{O(v)}$, both of which  contradict $I(v)^+\prec e \prec O(v)^-$.

Now we show $(Q_2)\Longrightarrow(U_3)$. Assume $I(v_1)\cap \overline{I(v_2)}\neq \emptyset$, then there exists $e\in I(v_1)$ such that $I(v_2)^-\prec e \prec I(v_2)^+$. Clearly, $v_2=t(I(v_2)^-)=t(I(v_2)^+)$, then by $(Q_2)$, we have $I(v_1)=I(t(e))\subseteq \overline{I(v_2)}$. Similarly, $O(v_1)\cap \overline{O(v_2)}\neq \emptyset$ means that there exists $e\in O(v_1)$ such that $O(v_2)^-\prec e \prec O(v_2)^+$. By $v_2=s(O(v_2)^-)=s(O(v_2)^+)$ and $(Q_2)$, we have $O(v_1)=O(s(e))\subseteq \overline{O(v_2)}$.

$(\Longrightarrow)$. Assume $e_1\prec e\prec e_2$ and $e_1$, $e_2$ are adjacent, then there are three cases.

\noindent\textbf{Case 1:} $t(e_1)=t(e_2)=v$. In this case, $e\in I(t(e))\cap\overline{I(v)}$, by $(U_3)$, we have $I(t(e))\subseteq\overline{I(v)}$.

\noindent\textbf{Case 2:} $s(e_1)=s(e_2)=v$. In this case, $e\in O(s(e))\cap\overline{O(v)}$, by $(U_3)$, we have $I(s(e))\subseteq\overline{O(v)}$.

\noindent\textbf{Case 3:} $t(e_1)=s(e_2)=v$. In this case, $e\in \overline{E(v)}$, then by $(U_2)$, we have either $e\in \overline{I(v)}$ or $e\in \overline{O(v)}$, which reduces the proof to case 1 or case 2, respectively.
\end{proof}

 The following theorem is the main result of \cite{[LY19]}(Theorem $6.1$), which says that upward planar orders combinatorially  characterize upward planar drawings.

\begin{thm}
Any UPO-graph has a unique upward planar drawing up to planar isotopy; and conversely, there is an upward planar order on $E(G)$ for any upward plane graph $G$.
\end{thm}

\section{Admissible condition}
In this section, we recall the notion of a progressive graph, which is an acyclic directed graph with boundary,  and introduce the notion of  admissible condition for upward planar orders on progressive graphs with the aim to give a combinatorial characterization of progressive plane graphs.

The following notion was introduced in \cite{[JS91]}.
\begin{defn}
An acyclic directed graph $G$ with a distinguished subset $\sigma$ of leaves (vertices of degree one) is called a \textbf{progressive graph}, denoted by $G^\sigma$ or $(G,\sigma)$.
\end{defn}
The distinguished subset $\sigma$ is called the \textbf{boundary} of $G^\sigma$. A progressive graph $G^\sigma$ is an acyclic directed graph $G$ with boundary $\sigma$, and we write this fact as $\partial G^\sigma=\sigma$. Any acyclic directed graph can be naturally viewed as a progressive graph with $\sigma=\emptyset$.
The elements of $\sigma$ and $V(G)-\sigma$ are called \textbf{boundary vertices} and  \textbf{inner vertices}, respectively.
A boundary vertex $v\in \sigma$ is also called an \textbf{input vertex} if $v$ is a source, and is called an \textbf{output vertex} if $v$ is a sink. An edge starting from a boundary vertex is called an \textbf{input edge} of $G^\sigma$; and an edge ending with a boundary vertex is called an \textbf{output edge} of $G^\sigma$. The set of input edges is called the \textbf{domain} of $G^\sigma$, denoted by $I(G^\sigma)$. The set of output edges is called the \textbf{codomain} of $G^\sigma$, denoted by  $O(G^\sigma)$.  When the boundary $\sigma$ is clear or irrelevant, we are free to write $G$ instead of $G^\sigma$ and $(G,\sigma)$ for convenience.

As shown in Fig. $2$ and Example \ref{type},  a progressive plane graph is an upward planar drawing of a progressive graph in a plane box with the boundary condition that a vertex is a boundary vertex if and only if it is drawn on the horizontal boundaries.

An upward plane graph can be combinatorially characterized as a UPO-graph. To give a similar combinatorial characterization of progressive plane graphs, we need the following notion.

\begin{defn}
An upward planar order $\preceq$ on a progressive graph $G^\sigma$ is called \textbf{admissible} if it satisfies the condition that for any inner vertex $v\in V(G)-\sigma$, $I(G^\sigma)\cap \overline{O(v)}=\emptyset$ and $O(G^\sigma)\cap \overline{I(v)}=\emptyset$.
\end{defn}

The upward planar order in Example \ref{type} is admissible. A progressive graph together with an admissible upward planar order is called an \textbf{admissible UPO-graph}, which serves as a combinatorial formulation of a progressive plane graph. The admissible condition can be understood as a combinatorial formulation of the boundary condition of a progressive plane graph.
\begin{ex}
The following shows an admissible UPO-graph (left) and a non-admissible UPO-graph (right), where the boundary vertices are omitted for convenience.
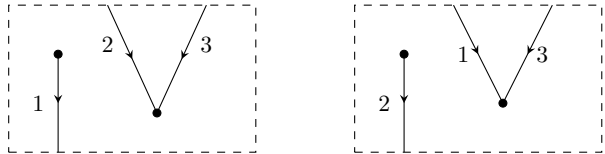
\begin{center}
\begin{tikzpicture}[scale=1.3]
	\draw[dashed]  (-4,4) rectangle (-1.5,2.5);
	\draw[dashed]  (-0.5,4) rectangle (2,2.5);
	\node at (-2.5,3.5) {};
	\node at (1,3.5) {};
	\node at (-3,4) {};
	\node at (-2,4) {};
	
	\draw[fill] (-2.5,2.9) circle [radius=0.04];
	\draw[fill] (1,3) circle [radius=0.04];
	\draw (-3,4) -- (-2.5,2.9)[postaction={decorate, decoration={markings,mark=at position .5 with {\arrow[black]{stealth}}}}];
	\draw (-2,4) -- (-2.5,2.9)[postaction={decorate, decoration={markings,mark=at position .5 with {\arrow[black]{stealth}}}}];
	\node at (0.5,4) {};
	\node at (1.5,4) {};
	\draw (0.5,4) -- (1,3)[postaction={decorate, decoration={markings,mark=at position .5 with {\arrow[black]{stealth}}}}];
	\draw (1.5,4) -- (1,3)[postaction={decorate, decoration={markings,mark=at position .5 with {\arrow[black]{stealth}}}}];
	\node at (-3.5,3.5) {};
	\node at (-3.5,2.5) {};
	\draw[fill] (-3.5,3.5) circle [radius=0.04];
	\draw (-3.5,3.5) -- (-3.5,2.5)[postaction={decorate, decoration={markings,mark=at position .5 with {\arrow[black]{stealth}}}}];
	
	\node at (0,3.5) {};
	\draw[fill] (0,3.5) circle [radius=0.04];
	\node at (0,2.5) {};
	\draw (0,3.5) -- (0,2.5)[postaction={decorate, decoration={markings,mark=at position .5 with {\arrow[black]{stealth}}}}];
	\node[scale=0.8] at (-3.7,3) {$1$};
	\node[scale=0.8] at (-3,3.6) {$2$};
	\node[scale=0.8] at (-2,3.6) {$3$};
	\node[scale=0.8] at (0.6,3.5) {$1$};
	\node[scale=0.8] at (1.4,3.5) {$3$};
	\node[scale=0.8] at (-0.2,3) {$2$};
\end{tikzpicture}
\end{center}
\end{ex}

Notice that, in graphical calculi,  progressive graphs can have isolated vertices. In these cases,  we can make our combinatorial characterization still valid by changing isolated vertices into isolated virtual edges.

$$
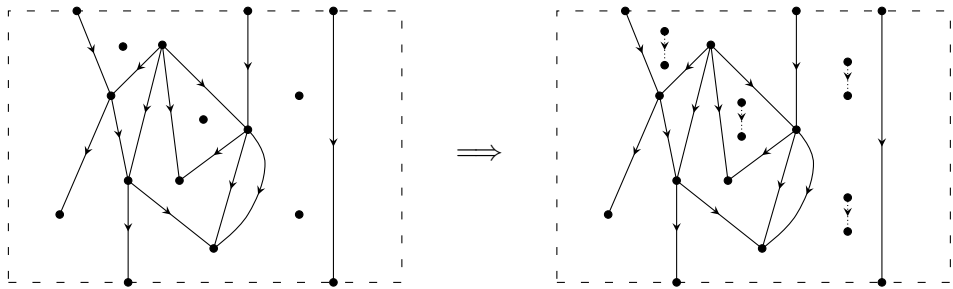
\begin{matrix}\begin{matrix}
\begin{tikzpicture}[scale=0.45]

\draw [loosely dashed] (-3,1.5) rectangle (8.5,-6.5);
\node [above](v1) at (-1,1.5) {};
\node (v2) at (0,-1) {};
\node (v6) at (-1.5,-4.5) {};
\node [left] at (-1.5,-4.5) {};
\node (v3) at (1.5,0.5) {};
\node (v8) at (4,-2) {};
\node (v4) at (0.5,-3.5) {};
\node [below](v5) at (0.5,-6.5) {};
\node [above](v7) at (4,1.5) {};
\draw  (-1,1.5)  -- (0,-1)[postaction={decorate, decoration={markings,mark=at position .5 with {\arrow[black]{stealth}}}}];
\draw   (1.5,0.5) -- (0,-1)[postaction={decorate, decoration={markings,mark=at position .5 with {\arrow[black]{stealth}}}}];
\draw   (1.5,0.5)--  (0.5,-3.5)[postaction={decorate, decoration={markings,mark=at position .5 with {\arrow[black]{stealth}}}}];
\draw  (0.5,-3.5) -- (v5)[postaction={decorate, decoration={markings,mark=at position .5 with {\arrow[black]{stealth}}}}];
\draw  (0,-1) -- (-1.5,-4.5)[postaction={decorate, decoration={markings,mark=at position .5 with {\arrow[black]{stealth}}}}];
\draw  (0,-1) -- (0.5,-3.5)[postaction={decorate, decoration={markings,mark=at position .5 with {\arrow[black]{stealth}}}}];
\draw  (4,1.5)-- (4,-2)[postaction={decorate, decoration={markings,mark=at position .5 with {\arrow[black]{stealth}}}}];
\draw   (1.5,0.5)-- (4,-2)[postaction={decorate, decoration={markings,mark=at position .5 with {\arrow[black]{stealth}}}}];
\node (v9) at (3,-5.5) {};
\draw  (0.5,-3.5) -- (3,-5.5)[postaction={decorate, decoration={markings,mark=at position .5 with {\arrow[black]{stealth}}}}];
\draw  (4,-2) -- (3,-5.5)[postaction={decorate, decoration={markings,mark=at position .5 with {\arrow[black]{stealth}}}}];
\node (v10) at (2,-3.5) {};
\draw   (1.5,0.5) -- (2,-3.5)[postaction={decorate, decoration={markings,mark=at position .5 with {\arrow[black]{stealth}}}}];
\draw  (4,-2) -- (2,-3.5)[postaction={decorate, decoration={markings,mark=at position .5 with {\arrow[black]{stealth}}}}];

\draw [fill](v2) circle [radius=0.11];
\draw [fill](v3) circle [radius=0.11];
\draw [fill](v4) circle [radius=0.11];
\draw [fill](v6) circle [radius=0.11];
\draw [fill](v8) circle [radius=0.11];
\draw [fill](v9) circle [radius=0.11];
\draw [fill](v10) circle [radius=0.11];

\node at (-1,0.5) {};
\node at (4.5,0) {};
\node at (0,-5.5) {};
\node at (-1.5,-3) {};

\draw(6.5,1.5)--(6.5,-6.5)[postaction={decorate, decoration={markings,mark=at position .5 with {\arrow[black]{stealth}}}}];
\node [above]at(6.5,1.5) {};
\node [below]at (6.5,-6.5) {};
\node at (7,-3) {};

\draw [fill] (-1,1.5) circle [radius=0.11];
\draw [fill](0.5,-6.5) circle [radius=0.11];
\draw [fill](4,1.5) circle [radius=0.11];
\draw [fill](6.5,-6.5) circle [radius=0.11];
\draw [fill](6.5,1.5) circle [radius=0.11];

\draw [fill](5.5,-1) circle [radius=0.11];
\draw [fill](2.7,-1.7) circle [radius=0.11];

\draw [fill](0.35,0.45) circle [radius=0.11];

\draw  plot[smooth, tension=.7] coordinates {(4,-2) (4.5,-3) (4,-4.5) (3,-5.5)}[postaction={decorate, decoration={markings,mark=at position .5 with {\arrow[black]{stealth}}}}];
\node at (5.5,-4.5) {};
\draw [fill](5.5,-4.5) circle [radius=0.11];
\end{tikzpicture}
\end{matrix}&&\begin{matrix}\Longrightarrow\end{matrix}&&\begin{matrix}
\begin{tikzpicture}[scale=0.45]

\draw [loosely dashed] (-3,1.5) rectangle (8.5,-6.5);
\node [above](v1) at (-1,1.5) {};
\node (v2) at (0,-1) {};
\node (v6) at (-1.5,-4.5) {};
\node [left] at (-1.5,-4.5) {};
\node (v3) at (1.5,0.5) {};
\node (v8) at (4,-2) {};
\node (v4) at (0.5,-3.5) {};
\node [below](v5) at (0.5,-6.5) {};
\node [above](v7) at (4,1.5) {};
\draw  (-1,1.5)  -- (0,-1)[postaction={decorate, decoration={markings,mark=at position .5 with {\arrow[black]{stealth}}}}];
\draw   (1.5,0.5) -- (0,-1)[postaction={decorate, decoration={markings,mark=at position .5 with {\arrow[black]{stealth}}}}];
\draw   (1.5,0.5)--  (0.5,-3.5)[postaction={decorate, decoration={markings,mark=at position .5 with {\arrow[black]{stealth}}}}];
\draw  (0.5,-3.5) -- (v5)[postaction={decorate, decoration={markings,mark=at position .5 with {\arrow[black]{stealth}}}}];
\draw  (0,-1) -- (-1.5,-4.5) node (v19) {}[postaction={decorate, decoration={markings,mark=at position .5 with {\arrow[black]{stealth}}}}];
\draw  (0,-1) node (v16) {} -- (0.5,-3.5)[postaction={decorate, decoration={markings,mark=at position .5 with {\arrow[black]{stealth}}}}];
\draw  (4,1.5)-- (4,-2)[postaction={decorate, decoration={markings,mark=at position .5 with {\arrow[black]{stealth}}}}];
\draw   (1.5,0.5)-- (4,-2)[postaction={decorate, decoration={markings,mark=at position .5 with {\arrow[black]{stealth}}}}];
\node (v9) at (3,-5.5) {};
\draw  (0.5,-3.5) -- (3,-5.5)[postaction={decorate, decoration={markings,mark=at position .5 with {\arrow[black]{stealth}}}}];
\draw  (4,-2) -- (3,-5.5) node (v22) {}[postaction={decorate, decoration={markings,mark=at position .5 with {\arrow[black]{stealth}}}}];
\node (v10) at (2,-3.5) {};
\draw   (1.5,0.5) node (v18) {} -- (2,-3.5)[postaction={decorate, decoration={markings,mark=at position .5 with {\arrow[black]{stealth}}}}];
\draw  (4,-2) -- (2,-3.5) node (v21) {}[postaction={decorate, decoration={markings,mark=at position .5 with {\arrow[black]{stealth}}}}];

\draw [fill](v2) circle [radius=0.11];
\draw [fill](v3) circle [radius=0.11];
\draw [fill](v4) circle [radius=0.11];
\draw [fill](v6) circle [radius=0.11];
\draw [fill](v8) circle [radius=0.11];
\draw [fill](v9) circle [radius=0.11];
\draw [fill](v10) circle [radius=0.11];

\node at (-1,0.5) {};
\node at (4.5,0) {};
\node at (0,-5.5) {};
\node at (-1.5,-3) {};

\draw(6.5,1.5)--(6.5,-6.5)[postaction={decorate, decoration={markings,mark=at position .5 with {\arrow[black]{stealth}}}}];
\node [above]at(6.5,1.5) {};
\node [below]at (6.5,-6.5) {};
\node at (7,-3) {};

\draw [fill] (-1,1.5) circle [radius=0.11];

\draw [fill](0.5,-6.5) circle [radius=0.11];
\draw [fill](4,1.5) circle [radius=0.11];
\draw [fill](6.5,-6.5) circle [radius=0.11];
\draw [fill](6.5,1.5) circle [radius=0.11];

\draw [fill](5.5,-1) circle [radius=0.11];
\draw [fill](5.5,0) circle [radius=0.11];
\draw [densely dotted]  (5.5,0)-- (5.5,-1)[postaction={decorate, decoration={markings,mark=at position .55 with {\arrow[black]{stealth}}}}];

\draw [fill](2.4,-1.2) circle [radius=0.11];
\draw [fill](2.4,-2.2) circle [radius=0.11];
\draw [densely dotted]  (2.4,-1.2)-- (2.4,-2.2)[postaction={decorate, decoration={markings,mark=at position .55 with {\arrow[black]{stealth}}}}];

\draw [fill](0.15,0.9) circle [radius=0.11];
\draw [fill](0.15,-0.1) circle [radius=0.11];
\draw [densely dotted]  (0.15,0.9)-- (0.15,-0.1)[postaction={decorate, decoration={markings,mark=at position .55 with {\arrow[black]{stealth}}}}];

\draw  plot[smooth, tension=.7] coordinates {(4,-2) (4.5,-3) (4,-4.5) (3,-5.5)}[postaction={decorate, decoration={markings,mark=at position .5 with {\arrow[black]{stealth}}}}];
\node at (5.5,-4.5) {};
\node at (5.5,-3.5) {};
\draw [fill](5.5,-5) circle [radius=0.11];
\draw [fill](5.5,-4) circle [radius=0.11];
\draw [densely dotted]  (5.5,-4)-- (5.5,-5)[postaction={decorate, decoration={markings,mark=at position .55 with {\arrow[black]{stealth}}}}];

\end{tikzpicture}
\end{matrix}
\end{matrix}
$$

We list the correspondence between geometrical and combinatorial notions in this paper as follows.
\begin{center}
\begin{tikzpicture}[scale=3.5]
\node at (0.1,-3.3) {Geometry};
\node at (0.1,-3.7) {Combinatorics};

\draw  (-0.3,-3.1) rectangle (0.5,-3.5) node (v1) {};
\draw  (-0.3,-3.5) rectangle (0.5,-3.9);

\draw  (0.5,-3.1) rectangle (1.9,-3.5) node (v2) {};
\draw  (v1) rectangle (1.9,-3.9);

\node at (1.2,-3.3) {(boxed) upward plane graph};
\node at (1.2,-3.7) {UPO-graph};
\draw  (1.9,-3.1) rectangle (2.9,-3.5) node (v3) {};
\draw  (v2) rectangle (2.9,-3.9);
\node at (2.4,-3.3) {boundary condition};
\node at (2.4,-3.7) {admissible condition};

\draw  (2.9,-3.1) rectangle (4.1,-3.5);
\node at (3.5,-3.3) {progressive plane graph};
\draw  (v3) rectangle (4.1,-3.9);
\node at (3.5,-3.7) {admissible UPO-graph};
\end{tikzpicture}
\end{center}

\section{Composition of upward planar orders}

Given a progressive graph $(G_1,\sigma_1)$ with an admissible upward planar order $\prec_1$ and $n$ output edges $o_1\prec_1\cdots\prec_1 o_n$, depicted as
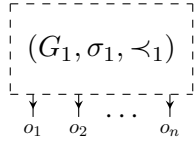
\begin{center}
\begin{tikzpicture}[scale=0.6]

\draw [dashed] (-1,2.5) -- (-1,0.5);
\draw [dashed] (-1,2.5) --(3,2.5);
\draw [dashed] (3,2.5) -- (3,0.5);
\draw [dashed] (-1,0.5) -- (3,0.5);

\node [scale=0.7][below] at (-0.5,0) {$o_1$};
\node (v12) at (0.5,0) {};
\node [scale=0.7][below] at (0.5,0) {$o_2$};
\node (v14) at (2.5,0) {};
\node [scale=0.7][below] at (2.5,0) {$o_n$};
\draw  (-0.5,0.5) --  (-0.5,0)[postaction={decorate, decoration={markings,mark=at position .65 with {\arrow[black]{stealth}}}}];
\draw  (0.5,0.5) -- (0.5,0)[postaction={decorate, decoration={markings,mark=at position .65 with {\arrow[black]{stealth}}}}];
\draw  (2.5,0.5) -- (2.5,0)[postaction={decorate, decoration={markings,mark=at position .65 with {\arrow[black]{stealth}}}}];
\node at (1,1.5) {$(G_1,\sigma_1,\prec_1)$};
\node at (1.5,0) {$\cdots$};
\end{tikzpicture}
\end{center}
and a progressive graph $(G_2,\sigma_2)$ with an admissible upward planar order $\prec_2$ and $n$ input edges $i_1\prec_2\cdots\prec_2 i_n$, depicted as
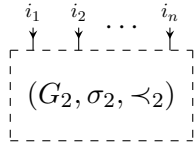
\begin{center}
\begin{tikzpicture}[scale=0.6]

\node (v5) at (-1,-1.5) {};
\node (v6) at (3,-1.5) {};
\node (v7) at (-1,-3.5) {};
\node (v8) at (3,-3.5) {};

\draw [dashed] (-1,-1.5) -- (3,-1.5);
\draw [dashed] (-1,-1.5) -- (-1,-3.5);
\draw [dashed] (3,-1.5) -- (3,-3.5);
\draw [dashed] (-1,-3.5) -- (3,-3.5);

\node [scale=0.7][above] at (-0.5,-1) {$i_1$};
\node (v17) at (0.5,-1) {};
\node[scale=0.7] [above] at (0.5,-1) {$i_2$};
\node (v19) at (2.5,-1) {};
\node [scale=0.7][above] at (2.5,-1) {$i_n$};
\node (v10) at (-0.5,0) {};

\draw  (-0.5,-1) -- (-0.5,-1.5)[postaction={decorate, decoration={markings,mark=at position .5 with {\arrow[black]{stealth}}}}];
\draw  (0.5,-1) -- (0.5,-1.5)[postaction={decorate, decoration={markings,mark=at position .5 with {\arrow[black]{stealth}}}}];
\draw  (2.5,-1) -- (2.5,-1.5)[postaction={decorate, decoration={markings,mark=at position .5 with {\arrow[black]{stealth}}}}];

\node at (1,-2.5) {$(G_2, \sigma_2,\prec_2)$};

\node at (1.5,-1) {$\cdots$};
\end{tikzpicture}
\end{center}
we can compose them to get a new progressive graph $(G,\sigma)=(G_2\circ G_1, \sigma_2\vee\sigma_1)$ by identifying  the output edge $o_k$ of $(G_1,\sigma_1)$ with the input edge $i_k$ of $(G_2,\sigma_2)$ for each $k=1,2,\cdots,n$, as depicted in Fig. $7$,  where $\overline{e_k}$ $(k=1,2,\cdots,n)$ are newly formed edges by identifying $i_k$ and $o_k$.

\begin{center}
\begin{tikzpicture}[scale=0.6]

\draw [dashed] (-1,2.5) -- (-1,0.5);
\draw [dashed] (-1,2.5) --(3,2.5);
\draw [dashed] (3,2.5) -- (3,0.5);
\draw [dashed] (-1,0.5) -- (3,0.5);

\node (v9) at (-0.5,0.5) {};
\node (v13) at (2.5,0.5) {};
\node (v11) at (0.5,0.5) {};
\node (v16) at (-0.5,-1.5) {};
\node (v18) at (0.5,-1.5) {};
\node (v20) at (2.5,-1.5) {};
\node (v15) at (-0.5,-1) {};

\node (v10) at (-0.5,0) {};
\node [scale=0.7][left] at (-0.5,0) {$o_1$};
\node (v12) at (0.5,0) {};
\node [scale=0.7][left] at (0.5,0) {$o_2$};
\node (v14) at (2.5,0) {};
\node [scale=0.7][left] at (2.5,0) {$o_n$};

\node [scale=0.7][right] at (-0.5,-0.5) {$\overline{e_1}$};
\node [scale=0.7][right] at (0.5,-0.5) {$\overline{e_2}$};
\node [scale=0.7][right] at (2.5,-0.5) {$\overline{e_n}$};

\node at (1,1.5) {$(G_1,\sigma_1)$};

\node at (1.5,0) {$\cdots$};


\node (v5) at (-1,-1.5) {};
\node (v6) at (3,-1.5) {};
\node (v7) at (-1,-3.5) {};
\node (v8) at (3,-3.5) {};

\draw [dashed] (-1,-1.5) -- (3,-1.5);
\draw [dashed] (-1,-1.5) -- (-1,-3.5);
\draw [dashed] (3,-1.5) -- (3,-3.5);
\draw [dashed] (-1,-3.5) -- (3,-3.5);

\node [scale=0.7][left] at (-0.5,-1) {$i_1$};
\node (v17) at (0.5,-1) {};
\node[scale=0.7] [left] at (0.5,-1) {$i_2$};
\node (v19) at (2.5,-1) {};
\node [scale=0.7][left] at (2.5,-1) {$i_n$};
\node (v10) at (-0.5,0) {};

\draw  (-0.5,0.5) -- (-0.5,-1.5)[postaction={decorate, decoration={markings,mark=at position .5 with {\arrow[black]{stealth}}}}];
\draw  (0.5,0.5) -- (0.5,-1.5)[postaction={decorate, decoration={markings,mark=at position .5 with {\arrow[black]{stealth}}}}];
\draw  (2.5,0.5) -- (2.5,-1.5)[postaction={decorate, decoration={markings,mark=at position .5 with {\arrow[black]{stealth}}}}];

\node at (1,-2.5) {$(G_2,\sigma_2)$};

\node at (1.5,-1) {$\cdots$};
\end{tikzpicture}

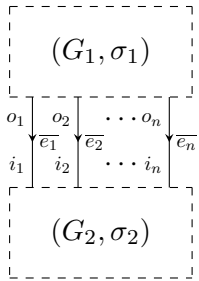
Figure $7$. Composition of $(G_1,\sigma_1)$ and $(G_2,\sigma_2)$
\end{center}

We mention that in the composition $G=G_2\circ G_1$, we remove all output vertices of $G_1$ and all input vertices of $G_2$; that the boundary $\sigma=\sigma_2\vartriangle\sigma_1$ of $G$ is the disjoint union of the set of input vertices of $G_1$ and the set of  output vertices of $G_2$; and that the set of inner vertices of $G$ is the disjoint union of the set of inner vertices of $G_1$ and the set of inner vertices of $G_2$.
For convenient, when necessary, we will freely identify $\overline{e_k}$ with $o_k$ or $i_k$; and moreover, we can freely view $E(G_1)$ and $E(G_2)$ as subsets of $E(G)$. Especially, we can say that $I(G^\sigma)=I(G_1^{\sigma_1})$ and $O(G^\sigma)=O(G_2^{\sigma_2})$.

To better describe the composite rule of upward plane orders, we introduce some notations.  For a linearly ordered set $(S,\prec)$, an \textbf{interval partition} is a partition $S=\bigsqcup_{i=1}^n E_i$ such that each block $E_i$ $(1\leq i\leq n)$ is an interval and $max\  E_i\prec min \ E_j$ for any $i<j$. For intervals $I_1$ and $I_2$, when $max \ I_1\prec min\ I_2$, we write $I_1\cup I_2=I_1\triangleleft I_2$, and when $max \ I_1\preceq min\  I_2$, we write $I_1\cup I_2=I_1\trianglelefteq I_2$.  We view such an interval partition as an ordered list of intervals and write $(S,\prec)=E_1\triangleleft E_2\triangleleft\cdots\triangleleft E_n$.

Since upward planar orders are linear orders on edges sets, then input edges and output edges can chop edge sets into disjoint intervals.
So  the edge posets of $G_1$ and $G_2$ have two interval partitions
$$(E(G_1), \prec_1)=Q_1\triangleleft\{o_1\}\triangleleft...\triangleleft Q_k\triangleleft\{o_k\} \triangleleft...\triangleleft Q_n\triangleleft\{o_n\}\triangleleft Q_{n+1}, $$
$$(E(G_2), \prec_2)=P_0\triangleleft\{i_1\}\triangleleft P_1\triangleleft...\triangleleft \{i_k\}\triangleleft P_k \triangleleft...\triangleleft\{i_n\}\triangleleft P_n,$$
with $Q_1=[1, o_1)$, $ Q_k=( o_{k-1}, o_k)$ $(2\leq k \leq n)$ and  $Q_{n+1}=(o_{n},+\infty]$; $P_0=[1,i_1)$, $P_k=(i_{k}, i_{k+1})$ $(1\leq k \leq n-1)$, $P_n=( i_n,+\infty]$, where $1$ and $+\infty$ denote the minimal element and maximal element, respectively. We call $Q_k$ $(k=1,2,\cdots,n+1)$ the $k$-th \textbf{basic interval} of $(G_1,\sigma_1,\prec_1)$ and $P_l$  $(l=0,1,\cdots,n)$ the $l$-th \textbf{basic interval} of $(G_2,\sigma_2,\prec_2)$. Basic intervals do not contains output edges of $G_1$ and input edges of $G_2$.

\begin{defn}
Given two admissible UPO-graphs $(G_1,\sigma_1, \prec_1)$ and $(G_2,\sigma_2,\prec_2)$ as above,
the \textbf{composed linear order} $\prec=\prec_2\circ\prec_1$ on $E(G_2\circ G_1)$ is defined as the shuffle-style  linear order
$$(E(G_2\circ G_1),\prec)=P_0\triangleleft Q_1\triangleleft\{\overline{e_1}\}\triangleleft P_1\triangleleft...\triangleleft Q_k\triangleleft\{\overline{e_k}\}\triangleleft P_k \triangleleft...\triangleleft Q_n\triangleleft\{\overline{e_n}\}\triangleleft P_n\triangleleft Q_{n+1}.$$
\end{defn}

As a direct consequence of the composition rule, we have the following lemma.
\begin{lem}\label{preserve}
The composition preserves the linear orders $\prec_1$ and $\prec_2$.
\end{lem}

Since the composition is of shuffle-style, it is not difficult to prove the following result.
\begin{thm}
The composition of upward planar orders is associative.
\end{thm}

The following is the main result of this paper, which says that the class of admissible UPO-graphs is closed under composition.

\begin{thm}\label{main}
The composition $(G,\sigma,\prec)=(G_2\circ G_1, \sigma_2\vartriangle\sigma_1, \prec_2\circ\prec_1)$ of any two admissible UPO-graphs $(G_1,\sigma_1,\prec_1)$ and $(G_2,\sigma_2,\prec_2)$ is  an admissible UPO-graph.
\end{thm}

\section{Proof of the main result}
To prove Theorem \ref{main}, we need to show that the composed linear order $\prec=\prec_2\circ\prec_1$ satisfies $(Q_1), (Q_2)$ and the admissible condition.

Same as in Section $4$, assume $(G_1,\sigma_1,\prec_1)$ is an admissible UPO-graph with $n$ output edges $o_1\prec_1\cdots\prec_1 o_n$; $(G_2,\sigma_2,\prec_2)$ is an admissible UPO-graph with $n$ input edges $i_1\prec_2\cdots\prec_2 i_n$; $\overline{e_k}$ $(k=1,2,\cdots,n)$ are newly formed edges by identifying $i_k$ and $o_k$.

\begin{prop}\label{q1}
The composed linear order $\prec$ satisfies $(Q_1)$.
\end{prop}
\begin{proof}
Assume $e_1\to e_2$ in $G=G_2\circ G_1$. There are three cases.

$(1)$ If $e_1\to e_2$ in $G_1$, then $e_1\prec_1 e_2$ in $G_1$, which, by Lemma \ref{preserve}, implies  that $e_1\prec e_2$ in $G$.

$(2)$ If $e_1\to e_2$ in $G_2$, then $e_1\prec_2 e_2$ in $G_2$, which,  by Lemma \ref{preserve}, implies  that $e_1\prec e_2$ in $G$.

$(3)$ If $e_1$ is an edge of $G_1$ and $e_2$ is an edge of $G_2$, then there must exist an $\overline{e_k}$ such that $e_1\to \overline{e_k}\to e_2$, which means that $e_1\to o_k$ in $G_1$ and $i_k\to e_2$ in $G_2$. Hence $e_1\prec_1 o_k$ in $G_1$ and $i_k\prec_2 e_2$ in $G_2$, which,  by Lemma \ref{preserve}, imply that $e_1\prec \overline{e_k}\prec e_2$ in $G$.
\end{proof}

\begin{prop}\label{ad}
The composed linear order $\prec$ satisfies the admissible condition.
\end{prop}
\begin{proof}
Let  $v$  be an arbitrary inner vertex of $G$, we want to show that $I(G)\cap\overline{O(v)}=\emptyset$ and $O(G)\cap\overline{I(v)}=\emptyset$ in $G$. By the definition of $G$, we have $I(G)=I(G_1)$, $O(G)=O(G_2)$. So  we only need to  show that $I(G_1)\cap\overline{O(v)}=\emptyset$ and $O(G_2)\cap\overline{I(v)}=\emptyset$ in $G$.

We have two cases.

$(1)$ $v$ is an inner vertex of $G_1$. Since $(G_1,\sigma_1,\prec_1)$ satisfies the admissible condition, then  $I(G_1)\cap \overline{O(v)}=\emptyset$ in $G_1$, which,  by Lemma \ref{preserve}, implies that $I(G_1)\cap \overline{O(v)}=\emptyset$ in $G$.

Similarly, by the admissible condition of $(G_1,\sigma_1,\prec_1)$, we have $O(G_1)\cap \overline{I(v)}=\emptyset$ in $G_1$, that is, $\{o_1,o_2,\cdots,o_n\}\cap \overline{I(v)}=\emptyset$ in $G_1$, then there must be a unique $Q_\alpha$ $(\alpha\in\{1,2,\cdots,n+1\})$ such that $\overline{I(v)}\subseteq Q_\alpha$  in $G_1$, which, by the composition rule, implies  $\overline{I(v)}\subseteq Q_\alpha$ is also true in $G$. Hence in $G$, $O(G_2)\cap\overline{I(v)}\subseteq O(G_2)\cap Q_\alpha= \emptyset$.

$(2)$ $v$ is an inner vertex of $G_2$. The proof is similar to case $(1)$. By the admissible condition of $(G_2,\sigma_2,\prec_2)$, we have $I(G_2)\cap\overline{O(v)}=\emptyset$ in $G_2$, that is, $\{i_1,i_2,\cdots,i_n\}\cap \overline{O(v)}=\emptyset$ in $G_2$, then there must be a unique $P_\beta$  $(\beta\in\{0,1,\cdots,n\})$ such that $\overline{O(v)}\subseteq P_\beta$ in $G_2$. Hence in $G$, $I(G_1)\cap\overline{O(v)}\subseteq I(G_1)\cap P_\beta=\emptyset$.

Also by the admissible condition of $(G_2,\sigma_2,\prec_2)$, we have $O(G_2)\cap\overline{I(v)}=\emptyset$ in $G_2$, which, by
Lemma \ref{preserve}, implies that $O(G_2)\cap \overline{I(v)}=\emptyset$ in $G$.
\end{proof}

The following lemma shows that the composed linear order $\prec$ satisfies the $(U_2)$ condition.
\begin{lem}\label{lem2}
Let $v$ be a processive vertex of $G$. Then $O(v)^-=I(v)^++1$ under the composed linear order $\prec$.
\end{lem}
\begin{proof}
Since $v$ is processive, then it must be either an inner vertex of $G_1$ or an inner vertex of $G_2$. We only prove the former case, the proof of the latter case is similar.

In the former case, since $\prec_1$ is admissible, then $\overline{I(v)}\cap \{o_1,\cdots,o_n\}=\emptyset$, which implies that $\overline{I(v)}\subseteq Q_\alpha$ for a unique $\alpha\in \{1,\cdots,n\}$. Notice that $\prec_1$ is an upward planar order on $G_1$, then, by Theorem \ref{eq}, $\prec_1$ satisfies the $(U_2)$ condition, which says that $O(v)^-=I(v)^++1$ in $(E(G_1),\prec_1)$. Therefore, $\overline{I(v)}\triangleleft \{O(v)^-\}\subseteq Q_\alpha \triangleleft\{o_\alpha\}$, which implies that either $O(v)^-\in Q_\alpha$ or $O(v)^-=o_\alpha$ in $G_1$. Both cases, by the composition rule, imply that $O(v)^-=I(v)^++1$ in $(E(G),\prec)$.
\end{proof}

For convenience, we introduce some notations. Let $J_1=Q_1\triangleleft\{o_1\}$, $\cdots$, $J_n=Q_n\triangleleft\{o_n\}$, $J_{n+1}=Q_{n+1}$ and call them the \textbf{right closures} of $Q_1,\cdots, Q_n, Q_{n+1}$ in $G_1$, respectively. Let $I_0=P_0$, $I_1=\{i_1\}\triangleleft P_1$, $\cdots$, $I_{n}=\{i_n\}\triangleleft P_n$ and call them the \textbf{left closures} of $P_0, P_1, \cdots, P_n$ in $G_2$, respectively.
Then we have $J_k\cap I_k=\{\overline{e_k}\}$,  $J_k\trianglelefteq I_k=Q_k\triangleleft \{\overline{e_k}\}\triangleleft P_k$ for $k\in\{1,2,\cdots,n\}$ in $G$, and

 $$(E(G_1), \prec_1)=J_1\triangleleft...\triangleleft J_k\triangleleft...\triangleleft J_n\triangleleft J_{n+1}, $$
$$(E(G_2), \prec_2)=I_0\triangleleft I_1\triangleleft...\triangleleft I_k \triangleleft...\triangleleft I_n,$$
$$(E(G), \prec)=I_0\triangleleft \Big(J_1\trianglelefteq I_1\Big)\triangleleft...\triangleleft \Big(J_k\trianglelefteq I_k\Big) \triangleleft...\triangleleft \Big(J_n \trianglelefteq I_n\Big) \triangleleft J_{n+1}.$$

\begin{prop}\label{q2}
The composed  linear order $\prec$ satisfies $(Q_2)$.
\end{prop}
\begin{proof}
We want to prove that if $e_1\prec e\prec e_2$ and $e_1$, $e_2$ are adjacent, then
\begin{equation*}
\begin{cases}
I(t(e))\subseteq \overline{I(v)},&  \text{if}\ t(e_1)=t(e_2)=v;\\
O(s(e))\subseteq \overline{O(v)},&  \text{if}\ s(e_1)=s(e_2)=v;\\
\text{either}\ I(t(e))\subseteq \overline{I(v)}\  \text{or}\ O(s(e))\subseteq \overline{O(v)},&  \text{if}\ t(e_1)=s(e_2)=v.
\end{cases}
\end{equation*}
We will prove this in three cases.

\textbf{Case 1.} $t(e_1)=t(e_2)=v$. We have four subcases.

 $(1.1)$ $e_1,e_2, e\in E(G_1)$. In this case, by the composition rule and  Lemma \ref{preserve}, $I(t(e))\subseteq \overline{I(v)}$ in  $G_1$ implies $I(t(e))\subseteq \overline{I(v)}$ in $G$.
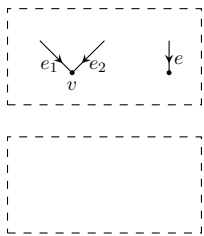
\begin{center}
	\begin{tikzpicture}[scale=0.85]
		\draw[dashed]  (-2,2.5) rectangle (1,1);
		\draw[dashed]  (-2,0.5) rectangle (1,-1);
		\node at (-1,1.5) {};
		\node at (-1.5,2) {};
		\node at (-0.5,2) {};
		\node at (0.5,1.5) {};
		\draw[fill] (-1,1.5) circle [radius=0.03];
		\draw[fill] (0.5,1.5) circle [radius=0.03];
		\draw (-1.5,2) -- (-1,1.5)[postaction={decorate, decoration={markings,mark=at position .7 with {\arrow[black]{stealth}}}}];;
		\draw (-0.5,2) -- (-1,1.5)[postaction={decorate, decoration={markings,mark=at position .7 with {\arrow[black]{stealth}}}}];;
		\draw (0.5,2) -- (0.5,1.5)[postaction={decorate, decoration={markings,mark=at position .7 with {\arrow[black]{stealth}}}}];;
		\node at (0.5,2) {};
		\node[scale=0.68] at (-1.35,1.6) {$e_1$};
		\node[scale=0.68] at (-0.6,1.6) {$e_2$};
		\node[scale=0.68] at (0.65,1.7) {$e$};
\node[scale=0.68] at (-1,1.3) {$v$};
	\end{tikzpicture}
   \end{center}

$(1.2)$ $e_1,e_2, e\in E(G_2)$.  In this case, by the composition rule and  Lemma \ref{preserve}, $I(t(e))\subseteq \overline{I(v)}$ in  $G_2$ implies $I(t(e))\subseteq \overline{I(v)}$ in $G$.
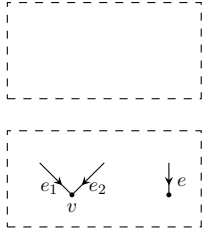
\begin{center}
	\begin{tikzpicture}[scale=0.85]
		\node at (-1,0) {};
		\node at (-0.5,-0.5) {};
		\node at (0,0) {};
		\node at (1,-0.5) {};
		\draw[fill] (-0.5,-0.5) circle [radius=0.03];
		\draw[fill] (1,-0.5) circle [radius=0.03];
		\draw (-1,0) -- (-0.5,-0.5)[postaction={decorate, decoration={markings,mark=at position .7 with {\arrow[black]{stealth}}}}];;
		\draw (0,0) -- (-0.5,-0.5)[postaction={decorate, decoration={markings,mark=at position .7 with {\arrow[black]{stealth}}}}];;
		\draw (1,0) -- (1,-0.5)[postaction={decorate, decoration={markings,mark=at position .7 with {\arrow[black]{stealth}}}}];;
		\node at (1,0) {};
		\node[scale=0.68] at (-0.85,-0.4) {$e_1$};
		\node[scale=0.68] at (-0.1,-0.4) {$e_2$};
		\node[scale=0.68] at (1.2,-0.3) {$e$};
		\draw [dashed] (-1.5,0.5) rectangle (1.5,-1);
		\draw[dashed]  (-1.5,2.5) rectangle (1.5,1);
		\node[scale=0.68] at (-0.5,-0.7) {$v$};
	\end{tikzpicture}
\end{center}

$(1.3)$ $e_1,e_2 \in E(G_1)$ and $e\in E(G_2)$. In this case, $v$ must be an inner vertex of $G_1$. Since $\prec_1$ is admissible, then $\overline{I(v)}\cap\{o_1,\cdots, o_n\}=\emptyset$ in $G_1$, which implies that the interval $\overline{I(v)}\subseteq Q_\alpha$ for a unique $\alpha\in\{1,\cdots, n+1\}$ in $G_1$. Since $e\in E(G_2)$, then $e\in I_\beta$ for a unique $\beta\in \{0,\cdots, n\}$. By the composition rule, we have either $Q_\alpha\triangleleft I_\beta$ or $I_\beta\triangleleft Q_\alpha$ in $G$, both of which contradict $e_1\prec e\prec e_2$. Therefore, this subcase is impossible.

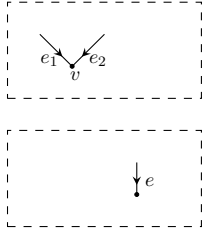
\begin{center}
	\begin{tikzpicture}[scale=0.85]
		\draw[dashed]  (-2,2.5) rectangle (1,1);
		\draw  [dashed](-2,0.5) rectangle (1,-1);
		\draw  [dashed](-1.5,2) rectangle (-1.5,2);
		\draw [dashed] (-1,1.5) rectangle (-1,1.5);
		\draw [dashed] (-0.5,2) rectangle (-0.5,2);
		\draw (-1.5,2) -- (-1,1.5)[postaction={decorate, decoration={markings,mark=at position .7 with {\arrow[black]{stealth}}}}];;
		\draw (-0.5,2) -- (-1,1.5)[postaction={decorate, decoration={markings,mark=at position .7 with {\arrow[black]{stealth}}}}];;
		\node[scale=0.68] at (-1.35,1.6) {$e_1$};
		\node[scale=0.68] at (-0.6,1.6) {$e_2$};
		\draw[fill] (-1,1.5) circle [radius=0.03];
		\draw  (0,0) rectangle (0,0);
		\draw  (0,-0.5) rectangle (0,-0.5);
		\draw[fill] (0,-0.5) circle [radius=0.03];
		\draw (0,0) -- (0,-0.5)[postaction={decorate, decoration={markings,mark=at position .7 with {\arrow[black]{stealth}}}}];;
		\draw  (0,-2.5) rectangle (0,-2.5);
		\node[scale=0.68] at (0.2,-0.3) {$e$};
			\node[scale=0.68] at (-0.95,1.35) {$v$};
	\end{tikzpicture}
\end{center}
$(1.4)$ $e\in E(G_1)$ and $e_1,e_2\in E(G_2)$. In this case,
if $e=\overline{e_\gamma}$, then $I(t(e))=\{e\}$, which, by $e_1\prec e\prec e_2$, implies that $I(t(e))\subset [e_1,e_2]\subseteq \overline{I(v)}$ in $G$. Otherwise, $t(e)$ is an inner vertex of $G_1$, then by the fact that
$\prec_1$ is admissible, we have $\overline{I(t(e))}\cap\{o_1,\cdots,o_n\}=\emptyset$ in $G_1$, which implies that $I(t(e))\subseteq \overline{I(t(e))} \subseteq Q_\gamma$ for a unique $\gamma\in\{1,\cdots,n+1\}$ in $G_1$. Clearly, $e_1\in I_\alpha$, $e_2\in I_\beta$ for unique $\alpha,\beta\in\{0,\cdots,n\}$ in $G_2$. By $e_1\prec e\prec e_2$ and the composition rule, we have $I_\alpha\triangleleft Q_\gamma \triangleleft I_\beta$, which implies that $I(t(e))\subseteq \overline{I(v)}$ in $G$.

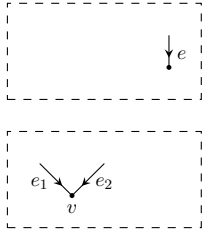
\begin{center}
	\begin{tikzpicture}[scale=0.85]
		\draw [dashed]  (-2.5,2) rectangle (0.5,0.5);
		\draw [dashed]  (-2.5,4) rectangle (0.5,2.5);
		\node at (0,3) {};
		\node at (0,3.5) {};
		\node at (-1.5,1) {};
		\node at (-2,1.5) {};
		\node at (-1,1.5) {};
		\draw[fill] (-1.5,1) circle [radius=0.03];
		\draw[fill] (0,3) circle [radius=0.03];
		\draw (0,3.5) -- (0,3)[postaction={decorate, decoration={markings,mark=at position .7 with {\arrow[black]{stealth}}}}];;
		\draw (-2,1.5) -- (-1.5,1)[postaction={decorate, decoration={markings,mark=at position .7 with {\arrow[black]{stealth}}}}];;
		\draw (-1,1.5) -- (-1.5,1)[postaction={decorate, decoration={markings,mark=at position .7 with {\arrow[black]{stealth}}}}];;
		\node[scale=0.68] at (0.2,3.2) {$e$};
		\node[scale=0.68] at (-1.5,0.8) {$v$};
		\node[scale=0.68] at (-2,1.2) {$e_1$};
		\node[scale=0.68] at (-1,1.2) {$e_2$};
	\end{tikzpicture}
\end{center}

\textbf{Case 2.}  $s(e_1)=s(e_2)=v$.  We have four subcases.

 $(2.1)$ $e_1,e_2, e\in E(G_1)$. Similar to case $(1.1)$, by the composition rule and  Lemma \ref{preserve}, $O(s(e))\subseteq \overline{O(v)}$ in  $G_1$ implies $O(s(e))\subseteq \overline{O(v)}$ in $G$.
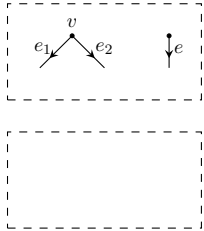
 \begin{center}
 \begin{tikzpicture}[scale=0.85]
		\node at (-1,2) {};
		\node at (-1.5,1.5) {};
		\node at (-0.5,1.5) {};
		\node at (0.5,2) {};
		\draw[fill] (-1,2) circle [radius=0.03];
		\draw[fill] (0.5,2) circle [radius=0.03];
		\draw (-1,2) -- (-0.5,1.5)[postaction={decorate, decoration={markings,mark=at position .7 with {\arrow[black]{stealth}}}}];;
		\draw (-1,2) -- (-1.5,1.5)[postaction={decorate, decoration={markings,mark=at position .7 with {\arrow[black]{stealth}}}}];;
		\draw (0.5,2) -- (0.5,1.5)[postaction={decorate, decoration={markings,mark=at position .7 with {\arrow[black]{stealth}}}}];;
		\node at (0.5,1.5) {};
		\node[scale=0.68] at (0.65,1.8) {$e$};
		\node[scale=0.68] at (-1.45,1.8) {$e_1$};
		\node[scale=0.68] at (-0.5,1.8) {$e_2$};
		\draw[dashed]  (-2,2.5) rectangle (1,1);
		\draw[dashed]  (-2,0.5) rectangle (1,-1);
		\node[scale=0.68] at (-1,2.2) {$v$};
	\end{tikzpicture}
\end{center}

$(2.2)$ $e_1,e_2, e\in E(G_2)$.  Similar to case $(1.2)$, by the composition rule and  Lemma \ref{preserve}, $O(s(e))\subseteq \overline{O(v)}$ in  $G_2$ implies $O(s(e))\subseteq \overline{O(v)}$ in $G$.

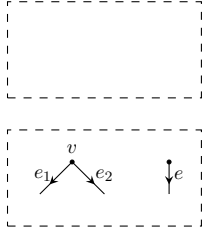
\begin{center}
	\begin{tikzpicture}[scale=0.85]
		\node at (-1,2) {};
		\node at (-1.5,1.5) {};
		\node at (-0.5,1.5) {};
		\node at (0.5,2) {};
		\draw[fill] (-1,2) circle [radius=0.03];
		\draw[fill] (0.5,2) circle [radius=0.03];
		\draw (-1,2) -- (-0.5,1.5)[postaction={decorate, decoration={markings,mark=at position .7 with {\arrow[black]{stealth}}}}];;
		\draw (-1,2) -- (-1.5,1.5)[postaction={decorate, decoration={markings,mark=at position .7 with {\arrow[black]{stealth}}}}];;
		\draw (0.5,2) -- (0.5,1.5)[postaction={decorate, decoration={markings,mark=at position .7 with {\arrow[black]{stealth}}}}];;
		\node at (0.5,1.5) {};
		\node[scale=0.68] at (0.65,1.8) {$e$};
		\node[scale=0.68] at (-1.45,1.8) {$e_1$};
		\node[scale=0.68] at (-0.5,1.8) {$e_2$};
		\draw [dashed] (-2,2.5) rectangle (1,1);
		\draw [dashed](-2,4.5) rectangle (1,3);
		\node[scale=0.68] at (-1,2.2) {$v$};
	\end{tikzpicture}
\end{center}

$(2.3)$ $e\in E(G_1)$ and $e_1,e_2\in E(G_2)$. Similar to case $(1.3)$, this case is impossible. In fact, $v$ must be an inner vertex of $G_2$. Since $\prec_2$ is admissible, then $\overline{O(v)}\cap\{i_1,\cdots, i_n\}=\emptyset$ in $G_2$, which implies that the interval $\overline{O(v)}\subseteq P_\alpha$ for a unique $\alpha\in\{0,\cdots, n\}$ in $G_2$. Since $e\in E(G_1)$, then $e\in J_\beta$ for a unique $\beta\in \{1,\cdots, n+1\}$. By the composition rule, we have either $P_\alpha\triangleleft J_\beta$ or $J_\beta\triangleleft P_\alpha$ in $G$, both of which contradict $e_1\prec e\prec e_2$.
\begin{center}
	\begin{tikzpicture}[scale=0.85]
		\draw [dashed]   (-2,2.5) rectangle (1,1);
		\draw [dashed]  (-2,0.5) rectangle (1,-1);
		\node at (-1,0) {};
		\node at (-1.5,-0.5) {};
		\node at (-0.5,-0.5) {};
		\draw (-1,0) -- (-1.5,-0.5)[postaction={decorate, decoration={markings,mark=at position .7 with {\arrow[black]{stealth}}}}];;
		\draw (-1,0) -- (-0.5,-0.5)[postaction={decorate, decoration={markings,mark=at position .7 with {\arrow[black]{stealth}}}}];;
		\draw[fill] (-1,0) circle [radius=0.03];
		\node[scale=0.68] at (-1.5,-0.2) {$e_1$};
		\node[scale=0.68] at (-0.5,-0.2) {$e_2$};
		\node at (0.5,2) {};
		\node at (0.5,1.5) {};
		\draw[fill] (0.5,2) circle [radius=0.03];
		\draw (0.5,2) -- (0.5,1.5)[postaction={decorate, decoration={markings,mark=at position .7 with {\arrow[black]{stealth}}}}];;
		\node[scale=0.68] at (0.7,1.7) {$e$};
		\node[scale=0.68] at (-1,0.2) {$v$};
	\end{tikzpicture}
\end{center}

$(2.4)$  $e_1,e_2 \in E(G_1)$ and $e\in E(G_2)$. Similar to case $(1.4)$,
if $e=\overline{e_\gamma}$, then $O(s(e))=\{e\}$, which, by $e_1\prec e\prec e_2$, implies that $O(s(e))\subset [e_1,e_2]\subseteq \overline{O(v)}$ in $G$. Otherwise, $s(e)$ is an inner vertex of $G_2$, then by the fact that
$\prec_2$ is admissible, we have $\overline{O(s(e))}\cap\{i_1,\cdots,i_n\}=\emptyset$ in $G_2$, which implies that $O(s(e))\subseteq \overline{O(s(e))} \subseteq P_\gamma$ for a unique $\gamma\in\{0,\cdots,n\}$ in $G_2$. Clearly, $e_1\in J_\alpha$, $e_2\in J_\beta$ for unique $\alpha,\beta\in\{0,\cdots,n\}$ in $G_1$. By $e_1\prec e\prec e_2$ and the composition rule, we have $J_\alpha\triangleleft P_\gamma \triangleleft J_\beta$, which implies that $O(s(e))\subseteq \overline{O(v)}$ in $G$.

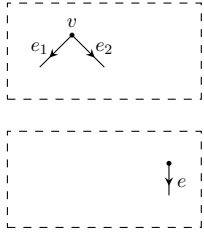
\begin{center}
	\begin{tikzpicture}[scale=0.85]
		\draw [dashed]  (-2.5,0.5) rectangle (-2.5,0.5) node (v1) {};
		\node at (-1.5,2) {};
		\node at (-2,1.5) {};
		\node at (-1,1.5) {};
		\node at (0.5,0) {};
		\node at (0.5,-0.5) {};
		\draw[fill] (-1.5,2) circle [radius=0.03];
		\draw[fill] (0,0) circle [radius=0.03];
		\draw (0,0) -- (0,-0.5)[postaction={decorate, decoration={markings,mark=at position .7 with {\arrow[black]{stealth}}}}];;
		\draw (-1.5,2) -- (-2,1.5)[postaction={decorate, decoration={markings,mark=at position .7 with {\arrow[black]{stealth}}}}];;
		\draw (-1.5,2) -- (-1,1.5)[postaction={decorate, decoration={markings,mark=at position .7 with {\arrow[black]{stealth}}}}];;
		\node[scale=0.68] at (-2,1.8) {$e_1$};
		\node[scale=0.68] at (-1,1.8) {$e_2$};
		\node[scale=0.68] at (0.2,-0.3) {$e$};
		\node[scale=0.68] at (-1.5,2.2) {$v$};
		\draw [dashed]  (-2.5,2.5) rectangle (0.5,1);
		\draw [dashed] (v1) rectangle (0.5,-1);
		\node at (0,0) {};
		\node at (0,-0.5) {};
	\end{tikzpicture}
\end{center}

\textbf{Case 3.} $t(e_1)=s(e_2)=v$. In this case, $v$ is a processive vertex of $G$. By Lemma \ref{lem2}, we have $[e_1,e_2]=[e_1,I(v)^+]\sqcup [O(v)^+,e_2]$ in $(E(G),\prec)$. Then $e_1\prec e\prec e_2$ implies that either $e_1\prec e\preceq I(v)^+$ or $O(v)^-\preceq e\prec e_2$. In the former case, by case $1$, we have $I(t(e))\subseteq \overline{I(v)}$. In the latter case, by case $2$, we have $O(s(e))\subseteq \overline{O(v)}$.

		\begin{center}
		\begin{tikzpicture}[scale=1.2]
			\node at (-0.5,2) {};
			\draw[fill] (-0.5,2) circle [radius=0.03];
			\node at (0,2.5) {};
			\node at (-1,2.5) {};
			\draw (0,2.7) -- (-0.5,2)[postaction={decorate, decoration={markings,mark=at position .5 with {\arrow[black]{stealth}}}}];;
			\draw (-1,2.7) -- (-0.5,2)[postaction={decorate, decoration={markings,mark=at position .5 with {\arrow[black]{stealth}}}}];;
			\node at (-1,1.5) {};
			\node at (0,1.5) {};
			\draw (-0.5,2) -- (-1,1.4)[postaction={decorate, decoration={markings,mark=at position .7 with {\arrow[black]{stealth}}}}];;
			\draw (-0.5,2) -- (0,1.4)[postaction={decorate, decoration={markings,mark=at position .7 with {\arrow[black]{stealth}}}}];;
			\node[scale=0.68] at (-1,2.5) {$e_1$};
			\node[scale=0.68] at (0.2,2.5) {$I(v)^{+}$};
			\node[scale=0.68] at (0.15,1.5) {$e_2$};
			\node[scale=0.68] at (-1.2,1.5) {$O(v)^{-}$};
			\node[scale=0.68] at (-0.65,2) {$v$};
		\end{tikzpicture}
	\end{center}
\end{proof}

Theorem \ref{main} is a direct consequence of Proposition \ref{q1}, Proposition \ref{ad} and Proposition \ref{q2}.
\section*{Acknowledgement}
This article is dedicated to Professor Ke Wu in Capital Normal University in celebration of his 80th birthday.

\par


\textbf{Xuexing Lu}\hfill \\  Email: xxlu@uzz.edu.cn \\ https://orcid.org/0000-0002-8066-1700

\textbf{Xue Dong}\hfill \\  Email: 3070973116@qq.com

\textbf{Yu Ye} \hfill \\ Email: yeyu@ustc.edu.cn


\begin{thebibliography}{1}
\bibitem{[BB91]}
G.~Bertolazzi, G.~Di~Battista.
\newblock On upward drawings of triconnected digraphs.
\newblock {\em Proc. 7th Annu. ACM Sympos. Comput. Geom.}, 272-280, 1991.


\bibitem{[BBLM94]}
G.~Bertolazzi, G.~Di~Battista, G.~Liotta, and C.Mannino.
\newblock Upward drawings of triconnected digraphs.
\newblock {\em Algorithmica}, 12:476-497, 1994.




\bibitem{[BT88]}
G.~Di~Battista and R.~Tamassia.
\newblock Algorithms for plane representations of acyclic digraphs.
\newblock {\em Theoret. Comput. Sci.}, 61(2-3):175-198, 1988.




\bibitem{[Ke87]}
 D.~Kelly.
\newblock Fundamentals of planar ordered sets.
\newblock {\em Discrete Mathematics}, 63(2):197-216, 1987.


\bibitem{[GT95]}
A.~Garg and R.~Tamassia.
\newblock Upward planarity testing.
\newblock {\em Order}, 12(2):109-133, 1995.

\bibitem{[LY19]}
X.~Lu and Y.~Ye.
\newblock Combinatorial characterization of upward planarity.
\newblock {\em Communications in Mathematics and Statistics}, 7:207-223, 2019.

\bibitem{[HLY19]}
S.~Hu, X.~Lu and Y.~Ye.
\newblock A graphical calculus for semi-groupal categories.
\newblock  {\em Applied Categorical Structures}, 27:163-197, 2019.

\bibitem{[JS88]}
A.~Joyal and R.~Street.
\newblock Planar diagrams and tensor algebra.
\newblock {\em Unpublished manuscript}, 1988.

\bibitem{[JS91]}
A.~Joyal and R.~Street.
\newblock The geometry of tensor calculus, I.
\newblock {\em Advances in mathematics}, 88:55-112, 1991.

\end{thebibliography}
\end{document}